\documentclass[12pt,oneside,english]{amsart}
\usepackage[T1]{fontenc}
\usepackage[latin9]{inputenc}
\usepackage{geometry}
\usepackage{amsthm}
\usepackage{amssymb}
\usepackage{setspace}
\usepackage{esint}
\makeatletter
\numberwithin{equation}{section}
\numberwithin{figure}{section}
\theoremstyle{plain}
\newtheorem{thm}{\protect\theoremname}
  \theoremstyle{definition}
  \newtheorem{defn}[thm]{\protect\definitionname}
  \theoremstyle{definition}
  \newtheorem{example}[thm]{\protect\examplename}
  \theoremstyle{plain}
  \newtheorem{lem}[thm]{\protect\lemmaname}
  \theoremstyle{remark}
  \newtheorem{rem}[thm]{\protect\remarkname}
  \theoremstyle{plain}
  \newtheorem{prop}[thm]{\protect\propositionname}

\makeatother

\usepackage{babel}
  \providecommand{\definitionname}{Definition}
  \providecommand{\examplename}{Example}
  \providecommand{\lemmaname}{Lemma}
  \providecommand{\propositionname}{Proposition}
  \providecommand{\remarkname}{Remark}
\providecommand{\theoremname}{Theorem}

\begin{document}

\title{Gelfand-Shilov Regularity of SG Boundary Value Problems.}

\author{Pedro T. P. Lopes}
\address{Departamento de Matem\'atica da Universidade Federal de S\~ao Carlos,  13565-905, S\~ao Carlos, SP, Brazil}
\email{dritao@yahoo.com; pedrolopes@dm.ufscar.br}
\thanks{The author was supported by FAPESP (Processo n$^{\circ}$ 2012/18198-9).}

%    General info

\date{\today}

\begin{abstract}
We show that the solutions of SG elliptic boundary value problems
defined on the complement of compact sets or on the half-space have
some regularity in Gelfand-Shilov spaces. The results are obtained
using classical results about Gevrey regularity of elliptic boundary
value problems and Calderón projectors techniques adapted to the SG
case. Recent developments about Gelfand-Shilov regularity of SG pseudo-differential
operators on $\mathbb{R}^{n}$ appear in an essential way.

\end{abstract}

\maketitle
\tableofcontents{}

In this paper, we use pseudo-differential operators and Gelfand-Shilov
spaces to obtain regularity results for elliptic boundary value problems
on two different classes of unbounded open sets $\Omega$: the complement
of compact sets of $\mathbb{R}^{n}$ - assuming some regularity on
the boundary $\Gamma=\partial\Omega$ - and on the half-space. We are
interested in the following boundary value problem:

\[
\begin{array}{c}
Pu=f,\,\,\,\mbox{in}\,\Omega\\
B^{j}u=g_{j},\,\,\,\mbox{on}\,\Gamma,\, j=1,2,...,r
\end{array},
\]
where the data $f$ and $g_{j}$ belong to appropriate Gelfand-Shilov
spaces and $P$ and $B_{j}$ are differential operators, whose symbols
belong to the SG class as in definition \ref{def:SG-symbols}. The
same class of boundary value problems were studied by C. Parenti,
H. O. Cordes, A. K. Erkip and E. Schrohe \cite{Cordeserkip,Erkipcommunications,Erkiplecturenotes,ErkipSchrohe,Parenti}.
Similar problems using Boutet de Monvel algebras were also studied
by E. Schrohe, D. Kapanadze and B. W. Schulze \cite{SchroheSGboutet,KapanadzeSchulzearticle,KapanadzeSchulze,Schulze}.

The class of SG differential and pseudo-differential operators were
studied by many authors. Its definition goes back at least to C. Parenti
\cite{Parenti} and H. O. Cordes \cite{Cordesglobalparametrix}. They
appear, for instance, in quantum mechanics equations, in scattering
theory problems, as in R. Melrose \cite[Chapter 6]{Melrosescattering},
and in recent generalizations of fifth and seventh order of the KdV
equation, as remarked by M. Cappiello, T. Gramchev and L. Rodino \cite{CappielloGramchevRodinogelpsuedolocalization}.
Elliptic boundary value problems on non-compact domains and manifolds
using SG pseudo-differential operators were also studied by the already
mentioned authors.

Recently M. Cappiello, T. Gramchev, F. Nicola and L. Rodino \cite{CappielloRodino,CappielloGramchevRodinogelpsuedolocalization,Rodinonicola}
have obtained, using pseudo-differential methods and Gelfand-Shilov
spaces, more precise regularity results for linear and semi-linear
SG elliptic problems in $\mathbb{R}^{n}$. They have applied these
results to prove the exponential decay of solutions of traveling waves
equations.

In our work, we study the same Gelfand-Shilov regularity, but for
the class of SG elliptic boundary value problems studied by C. Parenti
\cite{Parenti} and A. K. Erkip \cite{Erkipcommunications,Erkiplecturenotes}.
Our main results essentially state that if the data of the SG elliptic
boundary value problem are Gelfand-Shilov functions - or Gevrey on
the bounded boundary of a set - then so is the solution. In order
to do that, we first study and characterize the restrictions of Gelfand-Shilov
functions to the classes of unbounded domains in which we are interested.
The regularity results for boundary value problems on the complement
of compact sets are then easily obtained. They are given by Theorem
\ref{thm:Teorema principal no complemento de abertos} of Section
\ref{sec:Regularity-Results} and are a simple consequence of classical
results that can be found in J. Lions and E. Magenes \cite{lionsmagenes3}.
The regularity results on the half-space require a little more, as
the border is not compact. First we have to investigate the behavior
of the class of pseudo-differential operators defined by M. Cappiello,
T. Gramchev and L. Rodino \cite{CappielloRodino,CappielloGramchevRodinogelpsuedolocalization,Rodinonicola}
on the half-space, obtaining a kind of transmission property in the
sense of L. B. de Monvel \cite{boutetMonvel} for Gelfand-Shilov functions.
These results must be combined with Calderón projectors techniques
\cite{Calderon,Seeleysingular,Hormanderboudary,WlokaRowleyLawruk}
in order to obtain the desired regularity. Our main regularity result
for this case is given by Theorem \ref{thm:teorema principal semi plano}
of Section \ref{sec:Regularity-Results}.

\section{The Gelfand-Shilov space on open sets of $\mathbb{R}^{n}$.}

In this section, we define the Gelfand-Shilov spaces on open sets
$\Omega$ of $\mathbb{R}^{n}$. We show that, under certain assumptions,
our definition coincides with the restriction of the usual Gelfand-Shilov
functions to $\Omega$. Let us first explain some notation used in
this paper.

The open ball in $\mathbb{R}^{n}$ of radius $r>0$ and with center
at the origin is denoted by $B_{r}(0)$. We denote by $\mathbb{R}_{+}^{n}$
and $\mathbb{R}_{-}^{n}$ the set of points $x=(x',x_{n})\in\mathbb{R}^{n-1}\times\mathbb{R}=\mathbb{R}^{n}$
such that $x_{n}>0$ and $x_{n}<0$, respectively. The functions $r^{\pm}:\mathcal{D}'(\mathbb{R}^{n})\to\mathcal{D}'(\mathbb{R}_{\pm}^{n})$
are just the restrictions of the distributions. The extension by $0$
of a function defined in $\mathbb{R}_{-}^{n}$ to $\mathbb{R}^{n}$
or defined in $\mathbb{R}_{+}^{n}$ to $\mathbb{R}^{n}$ is denoted
by $e^{-}:L^{2}(\mathbb{R}_{-}^{n})\to L^{2}(\mathbb{R}^{n})$ and
$e^{+}:L^{2}(\mathbb{R}_{+}^{n})\to L^{2}(\mathbb{R}^{n})$, respectively.
The upper half-plane of $\mathbb{C}$ is denoted by $\mathbb{H}:=\left\{ z\in\mathbb{C};\,\mbox{Im}(z)\ge0\right\} $
and its interior is denoted by $\overset{\circ}{\mathbb{H}}:=\left\{ z\in\mathbb{C};\,\mbox{Im}(z)>0\right\} $.
We denote by $\left\langle .\right\rangle :\mathbb{R}^{n}\to\mathbb{R}$
the function $\left\langle x\right\rangle :=\sqrt{1+\left|x\right|^{2}}$
and the set $\{0,1,2,...\}$ of non negative integers by $\mathbb{N}_{0}$.
The space $\mathcal{B}(\mathbb{C}^{p},\mathbb{C}^{q})$ is just the
space of linear operators from $\mathbb{C}^{p}$ to $\mathbb{C}^{q}$.
The Gevrey functions in an open set $\Omega$ of order $\theta$ are
denoted by $G^{\theta}(\Omega)$. The ones with compact support contained
in $\Omega$ are denoted by $G_{c}^{\theta}(\Omega)$. As usual, $\mathcal{S}(\mathbb{R}^{n})$
is the Schwartz space of smooth functions whose derivatives are rapidly
decreasing. If $\Omega$ is an open set of $\mathbb{R}^{n}$, then
$\mathcal{S}(\Omega)$ denotes the set of restrictions of Schwartz
functions to $\Omega$. The main example is $\mathcal{S}(\mathbb{R}_{+}^{n})$.
For the Fourier transform, we use $\hat{u}(\xi)=\int e^{-ix\xi}u(x)dx$.
We denote by $\left(.,.\right)_{L^{2}(\mathbb{R}^{n})^{\oplus p}}:L^{2}(\mathbb{R}^{n})^{\oplus p}\times L^{2}(\mathbb{R}^{n})^{\oplus p}\to\mathbb{C}$
the scalar product $\left(u,v\right)_{L^{2}(\mathbb{R}^{n})^{\oplus p}}=\sum_{j=1}^{p}\int u_{j}(x)\overline{v_{j}(x)}dx$,
where $u=(u_{1},...,u_{p})$ and $v=(v_{1},...,v_{p})$. Finally,
we use the multi-index notation, that is, $\alpha=\left(\alpha_{1},...,\alpha_{n}\right)\in\mathbb{N}_{0}^{n}$,
$x^{\alpha}:=x_{1}^{\alpha_{1}}...x_{n}^{\alpha_{n}}$, $\partial_{x}^{\alpha}:=\partial_{x_{1}}^{\alpha_{1}}...\partial_{x_{n}}^{\alpha_{n}}$
and $D_{x}^{\alpha}:=D_{x_{1}}^{\alpha_{1}}...D_{x_{n}}^{\alpha_{n}}$,
where $D_{x_{j}}=-i\partial_{x_{j}}$. 

The next definition is a slight extension of the usual definition
of Gelfand-Shilov functions which can be found in \cite{gelfandshilov,Rodinonicola}:
\begin{defn}
Let $\mu>0$ and $\nu>0$ be constants such that $\mu+\nu\ge1$. Let
$\Omega\subset\mathbb{R}^{n}$ be an open set. The Gelfand-Shilov
space $\mathcal{S}_{\nu}^{\mu}(\Omega)$ is defined as the space of
functions $u\in C^{\infty}\left(\Omega\right)$ for which there are
constants $C>0$ and $D>0$ depending only on $u$ such that 
\[
\left|x^{\alpha}\partial_{x}^{\beta}u(x)\right|\le CD^{\left|\alpha\right|+\left|\beta\right|}\left(\alpha!\right)^{\nu}\left(\beta!\right)^{\mu},\,\forall\alpha,\beta\in\mathbb{N}_{0}^{n},\,\forall x\in\Omega.
\]

For each constant $D>0$, we define the subspace $\mathcal{S}_{\nu,D}^{\mu}(\Omega)\subset\mathcal{S}_{\nu}^{\mu}(\Omega)$
of the functions that satisfy the above estimate for the constant
$D$. This is a Banach space whose norm is given by
\[
\sup_{\alpha,\beta}\sup_{x\in\Omega}D^{-\left|\alpha\right|-\left|\beta\right|}\left(\alpha!\right)^{-\nu}\left(\beta!\right)^{-\mu}\left|x^{\alpha}\partial_{x}^{\beta}u(x)\right|.
\]

The space $\mathcal{S}_{\nu}^{\mu}(\Omega)$ is endowed with the topology
of inductive limit: $\mathcal{S}_{\nu}^{\mu}(\Omega)=\cup_{D>0}\mathcal{S}_{\nu,D}^{\mu}(\Omega)$.
The continuous linear functionals on $\mathcal{S}_{\nu}^{\mu}(\Omega)$
are denoted by $\mathcal{S}_{\nu}^{\mu'}(\Omega)$. 
\end{defn}
It is clear from the definition that $G_{c}^{\mu}(\Omega)\subset\mathcal{S}_{\nu}^{\mu}(\Omega)\subset G^{\mu}(\Omega)$.
Moreover, for bounded sets, $u\in\mathcal{S}_{\nu}^{\mu}(\Omega)$
if, and only if, 
\[
\left|\partial_{x}^{\beta}u(x)\right|\le CD^{\left|\beta\right|}\left(\beta!\right)^{\mu},\,\forall\beta\in\mathbb{N}_{0}^{n},\,\forall x\in\Omega.
\]

This means that, for $\mu\ge1$, $u\in\mathcal{S}_{\nu}^{\mu}(\Omega)$
if, and only if, $u$ is a Gevrey function of order $\mu$, and the
Gevrey estimates are uniform: the constants $C$ and $D$ hold for
all $\Omega$.

A function $u\in S_{\nu}^{\mu}(\Omega)$ has an exponential decay
of the form $\left|u(x)\right|\le Ce^{-\epsilon\left|x\right|^{\frac{1}{\nu}}}$,
for some $\epsilon>0$. Therefore the study of the regularity in these
spaces leads to a better understanding of the behavior of the solutions
at the infinity as well as the Gevrey regularity of the solutions.

A simple and useful remark, which will be used in Section \ref{sub:The behaviour of pseudo},
is that $u\in\mathcal{S}_{\nu}^{\mu}(\Omega)$ iff for every $m\in\mathbb{R}$,
there are constants $C>0$ and $D>0$ depending only on $u$ and $m$
such that 
\[
\left|x^{\alpha}\partial_{x}^{\beta}u(x)\right|\le CD^{\left|\alpha\right|+\left|\beta\right|}\left(\alpha!\right)^{\nu}\left(\beta!\right)^{\mu}\left\langle x\right\rangle ^{m},\,\forall\alpha,\beta\in\mathbb{N}_{0}^{n},\,\forall x\in\Omega.
\]

In general, if $\Omega\ne\mathbb{R}^{n}$, not every function in $\mathcal{S}_{\nu}^{\mu}(\Omega)$
is necessarily the restriction of a function in $\mathcal{S}_{\nu}^{\mu}(\mathbb{R}^{n})$,
as the following example shows.
\begin{example}
Let $u:\mathbb{R}_{+}\to\mathbb{C}$ be given by $u(x)=e^{-x}$. Hence
$u\in\mathcal{S}_{1}^{1}(\mathbb{R}_{+})$, but there is no function
$v\in\mathcal{S}_{1}^{1}(\mathbb{R})$ such that that $v(x)=u(x)$
for $x>0$. In fact, if $v\in\mathcal{S}_{1}^{1}(\mathbb{R})$, then
$v$ extends to a holomorphic function in the strip $\left\{ z\in\mathbb{C},-T\le\mbox{Im}(z)\le T\right\} $
\cite[Proposition 6.1.8.]{Rodinonicola}. Hence $v(z)=e^{-z}$ everywhere.
As $x\in\mathbb{R}\mapsto e^{-x}$ is not a function in $\mathcal{S}_{1}^{1}(\mathbb{R})\subset\mathcal{S}(\mathbb{R})$,
we obtain a contradiction.
\end{example}
For some situations, however, this is true. Let us study two situations:
The half-space and the complement of a compact  set. We start with
the half-space situation.
\begin{thm}
\label{thm:teo de extensao de func gelf shilov}Let $f\in\mathcal{S}_{\nu}^{\mu}(\mathbb{R}_{+}^{n})$,
$\mu>1$ and $\nu>0$. Then there is a function $g\in\mathcal{S}_{\nu}^{\mu}(\mathbb{R}^{n})$
such that $g(x)=f(x)$ for all $x\in\mathbb{R}_{+}^{n}$.
\end{thm}
In order to prove this theorem, we use the functions and results of
G. A. Džanašija \cite{Dzanasija}. Let us fix the constant $\mu>1$
and define functions $\left\{ a_{k},\, k=0,1,2,...\right\} $ and
$\left\{ b_{k},\, k=1,2,3,...\right\} $. 
\begin{defn}
\label{def:definicao ak e bk}Let $D\ge1$ and $r>0$. Let us assume
that $\frac{1}{2r}<\mu-1$. We define functions $b_{k}:\mathbb{R}\to\mathbb{R}$,
for $k\ge1$, as 
\[
b_{k}(t)=\left\{ \begin{array}{c}
0,\,\,\, t\in]-\infty,-\sigma_{k}[\\
\exp\left(\frac{-k\sigma_{k}^{4r}}{t^{2r}\left(\sigma_{k}+t\right)^{2r}}\right),\,\,\, t\in]-\sigma_{k},0[\\
0,\,\,\, t\in]0,\infty[
\end{array}\right.,
\]
where $\sigma_{k}=D^{-1}k^{-(\mu-1)}$.

We define $a_{k}:\mathbb{R}\to\mathbb{R}$ in the following way: For
$k\ge1$, we define 
\[
a_{k}(t)=\left\{ \begin{array}{c}
\frac{\int_{-\infty}^{t}b_{k}(y)dy}{\int_{-\infty}^{\infty}b_{k}(y)dy},\,\,\, t\in]-\infty,0[\\
\frac{\int_{-\infty}^{-t}b_{k}(y)dy}{\int_{-\infty}^{\infty}b_{k}(y)dy},\,\,\, t\in]0,\infty[
\end{array}\right..
\]

For $k=0$, we choose $a_{0}=a_{1}$.
\end{defn}
We note that for all $k\ge0$, $\mbox{supp}\left(a_{k}\right)\subset\left[-1,1\right]$,
$a_{k}(0)=1$ and $\left(\frac{d^{l}a_{k}}{dt^{l}}\right)\left(0\right)=0$,
for all $l>0$. The next lemma gives the properties of the functions
$a_{k}$ that we need. For its proof, we refer to \cite{Dzanasija}.
The constant $D\ge1$ will be chosen along the proof of Theorem \ref{thm:teo de extensao de func gelf shilov}.
\begin{lem}
\cite{Dzanasija} There is a constant $T>1$, depending only on $r>0$,
such that:

(i) If $k\le\alpha_{n}$, then

\[
\left|\partial_{x_{n}}^{\alpha_{n}}\left(a_{k}(x_{n})x_{n}^{k}\right)\right|\le2^{\alpha_{n}+1}\exp\left(ak\right)D^{-k}k^{-k\left(\mu-1\right)}T^{\alpha_{n}}D^{\alpha_{n}}\alpha_{n}^{\mu\alpha_{n}}.
\]

(ii) If $k>\alpha_{n}$, then 
\[
\left|\partial_{x_{n}}^{\alpha_{n}}\left(a_{k}(x_{n})x_{n}^{k}\right)\right|\le2^{\alpha_{n}+1}\exp\left(a\left(k+1\right)\right)D^{-k}k^{-k\left(\mu-1\right)}T^{\alpha_{n}}D^{\alpha_{n}}k^{\mu\alpha_{n}}.
\]

In the above expressions $a:=\frac{16^{2r}}{3^{2r}}$.
\end{lem}
Using these functions and the above lemma, we can prove Theorem \ref{thm:teo de extensao de func gelf shilov-1}
by actually providing an extension of the function $f$.
\begin{proof}
(of Theorem \ref{thm:teo de extensao de func gelf shilov}) As $f\in\mathcal{S}_{\nu}^{\mu}(\mathbb{R}_{+}^{n})$,
there is a constant $B>0$ such that
\begin{equation}
\left|x^{\alpha}\left(\partial_{x}^{\beta}f\right)(x)\right|\le B^{\left|\alpha\right|+\left|\beta\right|+1}\left(\alpha!\right)^{\nu}\left(\beta!\right)^{\mu},\,\forall\alpha,\beta\in\mathbb{N}_{0}^{n},\,\forall x\in\mathbb{R}_{+}^{n}.\label{eq:constante B}
\end{equation}

Let us choose and fix a constant $D\ge\max\left\{ 1,2Be^{a+1}\right\} $,
where $D$ is the constant used in Definition \ref{def:definicao ak e bk}. 

We define a function $h\in C^{\infty}\left(\overline{\mathbb{R}_{-}^{n}}\right)$
by

\[
h(x):=\sum_{k=0}^{\infty}\frac{1}{k!}a_{k}(x_{n})\left(\partial_{x_{n}}^{k}f\right)(x',0)x_{n}^{k},
\]
where $\left(\partial_{x_{n}}^{k}f\right)(x',0):=\lim_{x_{n}\to0^{+}}\left(\partial_{x_{n}}^{k}f\right)(x',x_{n})$.

We need to show that this series and its derivatives converge uniformly
to a function that satisfies the Gelfand-Shilov estimates on $\mathbb{R}_{-}^{n}$.
If this is the case, we see that
\[
\left(\partial_{x'}^{\alpha'}\partial_{x_{n}}^{\alpha_{n}}h\right)(x',0)=\sum_{k=0}^{\infty}\frac{1}{k!}\left.\partial_{x_{n}}^{\alpha_{n}}\left(a_{k}(x_{n})x_{n}^{k}\right)\right|_{x_{n}=0}\left(\partial_{x'}^{\alpha'}\partial_{x_{n}}^{k}f\right)(x',0)=\left(\partial_{x'}^{\alpha'}\partial_{x_{n}}^{\alpha_{n}}f\right)(x',0).
\]

Hence the function $g:\mathbb{R}^{n}\to\mathbb{C}$ defined as 
\[
g(x)=\left\{ \begin{array}{c}
f(x),\, x>0\\
h(x),\, x\le0
\end{array}\right.
\]
is such that $g\in\mathcal{S}_{\nu}^{\mu}(\mathbb{R}^{n})$ and $\left.g\right|_{\mathbb{R}_{+}^{n}}=f$.
This is the function we are looking for.

Let $k\le\alpha_{n}$. Then
\[
\left|x'^{\beta'}x_{n}^{\beta_{n}}\partial_{x'}^{\alpha'}\partial_{x_{n}}^{\alpha_{n}}\left(\frac{1}{k!}a_{k}(x_{n})\left(\partial_{x_{n}}^{k}f\right)(x',0)x_{n}^{k}\right)\right|=\frac{1}{k!}\left|x_{n}^{\beta_{n}}\partial_{x_{n}}^{\alpha_{n}}\left(a_{k}(x_{n})x_{n}^{k}\right)\left(x'^{\beta'}\partial_{x'}^{\alpha'}\partial_{x_{n}}^{k}f\right)(x',0)\right|\le
\]
\[
2^{\alpha_{n}+1}\exp\left(ak\right)D^{-k}k^{-k\left(\mu-1\right)}T^{\alpha_{n}}D^{\alpha_{n}}\alpha_{n}^{\mu\alpha_{n}}\frac{1}{k!}B^{\left|\alpha'\right|+\left|\beta'\right|+k+1}\left(\beta'!\right)^{\nu}\left(\alpha'!\right)^{\mu}\left(k!\right)^{\mu}\le
\]
\[
B^{\left|\alpha'\right|+\left|\beta'\right|+1}2^{\alpha_{n}+1}T^{\alpha_{n}}D^{\alpha_{n}}\left(\beta'!\right)^{\nu}\left(\alpha'!\right)^{\mu}\alpha_{n}^{\mu\alpha_{n}}\left(\frac{\exp\left(a\right)B}{D}\right)^{k}.
\]

As $D\ge2Be^{a}$, we obtain
\[
\sum_{k=1}^{\alpha_{n}}\left|x'^{\beta'}x_{n}^{\beta_{n}}\partial_{x'}^{\alpha'}\partial_{x_{n}}^{\alpha_{n}}\left(\frac{1}{k!}a_{k}(x_{n})\left(\partial_{x_{n}}^{k}f\right)(x',0)x_{n}^{k}\right)\right|\le
\]
\[
B^{\left|\alpha'\right|+\left|\beta'\right|+1}2^{\alpha_{n}+1}T^{\alpha_{n}}D^{\alpha_{n}}\left(\beta'!\right)^{\nu}\left(\alpha'!\right)^{\mu}\alpha_{n}^{\mu\alpha_{n}}\sum_{k=1}^{\alpha_{n}}\left(\frac{1}{2}\right)^{k}\le\tilde{C}_{1}\tilde{D}_{1}^{\left|\alpha\right|+\left|\beta'\right|}\left(\alpha!\right)^{\mu}\left(\beta'!\right)^{\nu},
\]
where $\tilde{C}_{1}>0$ and $\tilde{D}_{1}>0$ are constants that
do not depend on $\alpha$ and $\beta$.

For $k>\alpha_{n}$, we have
\[
\left|x'^{\beta'}x_{n}^{\beta_{n}}\partial_{x'}^{\alpha'}\partial_{x_{n}}^{\alpha_{n}}\left(\frac{1}{k!}a_{k}(x_{n})\left(\partial_{x_{n}}^{k}f\right)(x',0)x_{n}^{k}\right)\right|=\left|\frac{1}{k!}x_{n}^{\beta_{n}}\partial_{x_{n}}^{\alpha_{n}}\left(a_{k}(x_{n})x_{n}^{k}\right)\left(x'^{\beta'}\partial_{x'}^{\alpha'}\partial_{x_{n}}^{k}f\right)(x',0)\right|\le
\]
\[
2^{\alpha_{n}+1}\exp\left[a\left(k+1\right)\right]D^{-k}k^{-k\left(\mu-1\right)}T^{\alpha_{n}}D^{\alpha_{n}}k^{\mu\alpha_{n}}\frac{1}{k!}B^{\left|\alpha'\right|+\left|\beta'\right|+k+1}\left(\beta'!\right)^{\nu}\left(\alpha'!\right)^{\mu}\left(k!\right)^{\mu}\le
\]
\[
B^{\left|\alpha'\right|+\left|\beta'\right|+1}2^{\alpha_{n}+1}T^{\alpha_{n}}D^{\alpha_{n}}\left(\beta'!\right)^{\nu}\left(\alpha'!\right)^{\mu}k^{\mu\alpha_{n}}\exp\left(a\right)\left(\frac{\exp\left(a\right)B}{D}\right)^{k}.
\]

Using the inequality $e^{-a}\le a^{-d}d^{d}e^{-d}$, for $a>0$ and
$d>0$, we conclude that $k^{\mu\alpha_{n}}\le e^{k}\left(\mu\alpha_{n}\right)^{\mu\alpha_{n}}e^{-\mu\alpha_{n}}$.
As $D\ge2Be^{a+1}$, we obtain that
\[
\sum_{k=\alpha_{n}+1}^{\infty}\left|\frac{1}{k!}x_{n}^{\beta_{n}}\partial_{x_{n}}^{\alpha_{n}}\left(a_{k}(x_{n})x_{n}^{k}\right)\left(x'^{\beta'}\partial_{x'}^{\alpha'}\partial_{x_{n}}^{k}f\right)(x',0)\right|\le
\]
\[
\left(\exp\left(a\right)2B\right)B^{\left|\alpha'\right|+\left|\beta'\right|}\left(2TD\mu^{\mu}e^{-\mu}\right)^{\alpha_{n}}\left(\beta'!\right)^{\nu}\left(\alpha'!\right)^{\mu}\alpha_{n}^{\mu\alpha_{n}}\sum_{k=\alpha_{n}+1}^{\infty}\left(\frac{\exp\left(a+1\right)B}{D}\right)^{k}\le
\]
\[
\tilde{C}_{2}\tilde{D}_{2}^{\left|\alpha\right|+\left|\beta'\right|}\left(\alpha!\right)^{\mu}\left(\beta'!\right)^{\nu},
\]
where $\tilde{C}_{2}>0$ and $\tilde{D}_{2}>0$ are constants that
do not depend on $\alpha$ and $\beta$.\end{proof}
\begin{rem}
\label{rem:extensoes de funcoes gevrey}Precisely the same arguments
can be used to obtain extensions of Gevrey functions of order $\mu>1$:
If $f\in C^{\infty}(\mathbb{R}_{+}^{n})$ is a function such that
for all bounded sets $B\subset\mathbb{R}_{+}^{n}$ - not only compacts
- there are constants $C_{B}>0$ and $D_{B}>0$ such that
\[
\left|\partial_{x}^{\alpha}f(x)\right|\le C_{B}D_{B}^{\left|\alpha\right|}\left(\alpha!\right)^{\mu},\,\forall x\in B,
\]
then there is a Gevrey function $\tilde{f}$ of order $\mu$ defined
on $\mathbb{R}^{n}$ such that $f=\left.\tilde{f}\right|_{\mathbb{R}_{+}^{n}}$.
\end{rem}
The second situation in which we are interested is in the complement
of a bounded set. We need to be more precise about our assumptions.
\begin{defn}
\label{def:Gevrey submanifolds}Let $U$ be a bounded open set. We
say that its boundary $\Gamma=\partial U$ is a Gevrey $(n-1)$-manifold
of order $\Theta$, $U$ being locally on one side of $\Gamma$, if
for every $y\in\Gamma$, there is a bounded open set $\mathcal{O}\subset\mathbb{R}^{n}$,
$r_{y}>0$ and a Gevrey diffeomorphism $\psi:\mathcal{O}\to B_{r_{y}}(0)\subset\mathbb{R}^{n}$
of order $\Theta$ such that $\psi(U\cap\mathcal{O})=B_{r_{y}}(0)\cap\mathbb{R}_{-}^{n}$
and $\psi(\Gamma\cap\mathcal{O})=\left\{ x\in B_{r_{y}}(0);\, x_{n}=0\right\} $.
We also suppose that there exists a normal vector field $\nu$ on
$\Gamma$, such that the functions $\psi$ take $\nu$ to $\partial_{x_{n}}$.\end{defn}
\begin{thm}
\label{thm:teo de extensao de func gelf shilov-1}Let $\Omega=\mathbb{R}^{n}\backslash\overline{U}$,
where $U$ is a bounded open set, whose boundary $\Gamma=\partial U$
is a Gevrey $(n-1)$-manifold of order $\Theta$, $U$ being locally
on one side of $\Gamma$. If $f\in\mathcal{S}_{\nu}^{\mu}(\Omega)$,
$\mu>1$ and $\nu>0$, then there is a function $g\in\mathcal{S}_{\nu}^{\tilde{\mu}}(\mathbb{R}^{n})$,
$\tilde{\mu}=\max\left\{ \mu,\Theta\right\} $, such that $g(x)=f(x)$
for all $x\in\Omega$.\end{thm}
\begin{proof}
Let $\mathcal{O}_{1}$, ..., $\mathcal{O}_{N}$ be bounded open sets
such that $\Gamma=\cup_{j=1}^{N}\mathcal{O}_{j}$ and that there exist
Gevrey diffeomorphisms of order $\Theta$, $\psi_{j}:\mathcal{O}_{j}\to B_{r_{j}}(0)\subset\mathbb{R}^{n}$,
as in Definition \ref{def:Gevrey submanifolds}. Let $\mathcal{O}_{int}\subset U$
and $\mathcal{O}_{ext}\subset\Omega$ be open sets such that $\mathbb{R}^{n}=\mathcal{O}_{int}\cup\mathcal{O}_{ext}\cup_{j=1}^{N}\mathcal{O}_{j}$.
Let $\phi_{int}$, $\phi_{ext}$, $\phi_{1}$, ..., $\phi_{N}$ be
Gevrey functions of order $\Theta$ that form a partition of unity
subordinate to the open cover $\left\{ \mathcal{O}_{int},\mathcal{O}_{ext},\mathcal{O}_{1},\,...,\mathcal{O}_{N}\right\} $.

Then $f\circ\psi_{j}^{-1}:B_{r_{j}}(0)\to\mathbb{C}$ is a function
that satisfies, for all compact sets $K\subset B_{r_{j}}(0)$, the
estimates
\[
\left|\partial_{x}^{\alpha}\left(f\circ\psi_{j}^{-1}\right)(x)\right|\le C_{K}D_{K}^{\left|\alpha\right|}\left(\alpha!\right)^{\tilde{\mu}},\,\,\forall x\in K\cap\mathbb{R}_{+}^{n},
\]
where $\tilde{\mu}=\max\left\{ \mu,\Theta\right\} $.

Using Remark \ref{rem:extensoes de funcoes gevrey}, we conclude that
there is a Gevrey function $g_{j}:B_{r_{j}}(0)\to\mathbb{C}$ of order
$\tilde{\mu}$ such that $\left.g_{j}\right|_{B_{r_{j}}(0)\cap\mathbb{R}_{+}^{n}}=f\circ\psi_{j}^{-1}$.

Let us define $\tilde{f}:\mathbb{R}^{n}\to\mathbb{C}$ as 
\[
\tilde{f}(x)=f(x)\phi_{ext}(x)+\sum_{j=1}^{n}\left(g_{j}\circ\psi_{j}\right)(x)\phi_{j}(x).
\]

Hence $\tilde{f}$ extends $f$ and it is a Gevrey function of order
$\tilde{\mu}$ on $\mathbb{R}^{n}$. Let $R>0$ be such that $U\subset B_{R}(0)$.
Then
\[
\left|x^{\alpha}\left(\partial_{x}^{\beta}\tilde{f}\right)(x)\right|\le\left\{ \begin{array}{c}
CD^{\left|\alpha\right|+\left|\beta\right|}\left(\alpha!\right)^{\nu}\left(\beta!\right)^{\mu},\,\,\mbox{if}\, x\notin B_{R}(0)\\
\tilde{C}R^{\left|\alpha\right|}\tilde{D}^{\left|\beta\right|}\left(\beta!\right)^{\tilde{\mu}},\,\,\mbox{if}\, x\in\overline{B_{R}(0)}
\end{array}\right..
\]

This implies that $\tilde{f}\in S_{\nu}^{\tilde{\mu}}(\mathbb{R}^{n})$.
\end{proof}

\section{SG pseudo-differential operators.}

In this section, we recall the main properties of the SG calculus.
A recent detailed exposition can be found in F. Nicola and L. Rodino
\cite{Rodinonicola}. The calculus with symbols in $S_{\mu\nu}^{m_{1},m_{2}}(\mathbb{R}^{n}\times\mathbb{R}^{n})$
were originally, as far as we know, defined and studied by Cappiello
and Rodino in \cite{CappielloRodino}.
\begin{defn}
\emph{\label{def:SG-symbols}(SG symbols)} Let $m_{1}$ and $m_{2}$
belong to $\mathbb{R}$. We denote by $S^{m_{1},m_{2}}(\mathbb{R}^{n}\times\mathbb{R}^{n})$
the set of all functions $a\in C^{\infty}(\mathbb{R}^{n}\times\mathbb{R}^{n})$
satisfying, for all $\alpha$, $\beta\in\mathbb{N}_{0}^{n}$, the
estimates 
\[
\left|D_{x}^{\beta}D_{\xi}^{\alpha}a(x,\xi)\right|\le C_{\alpha\beta}\left\langle \xi\right\rangle ^{m_{1}-|\alpha|}\left\langle x\right\rangle ^{m_{2}-|\beta|},\,\,\,\forall(x,\xi)\in\mathbb{R}^{2n},
\]
where $C_{\alpha\beta}$ is a constant that depends on $a$, $\alpha$
and $\beta$. These functions are called SG symbols of class $(m_{1},m_{2})\in\mathbb{R}^{2}$. 

Similarly, we denote by $S_{\mu\nu}^{m_{1},m_{2}}(\mathbb{R}^{n}\times\mathbb{R}^{n})$,
where $\mu$ and $\nu$ are real numbers such that $\mu\ge1$ and
$\nu\ge1$, the subspace of $S^{m_{1},m_{2}}(\mathbb{R}^{n}\times\mathbb{R}^{n})$
defined as follows: $a\in S_{\mu\nu}^{m_{1},m_{2}}(\mathbb{R}^{n}\times\mathbb{R}^{n})$
if there are constants $C>0$ and $D>0$ depending only on $a$ such
that $C_{\alpha\beta}=CD^{\left|\alpha\right|+\left|\beta\right|}\left(\alpha!\right)^{\mu}\left(\beta!\right)^{\nu}$.

We denote by $S^{m_{1},m_{2}}(\mathbb{R}^{n}\times\mathbb{R}^{n},\mathcal{B}(\mathbb{C}^{p},\mathbb{C}^{q}))$
and $S_{\mu\nu}^{m_{1},m_{2}}(\mathbb{R}^{n}\times\mathbb{R}^{n},\mathcal{B}(\mathbb{C}^{p},\mathbb{C}^{q}))$
the classes of functions $a:\mathbb{R}^{n}\times\mathbb{R}^{n}\to\mathcal{B}(\mathbb{C}^{p},\mathbb{C}^{q})$,
where each entry of the matrix belongs to $S^{m_{1},m_{2}}(\mathbb{R}^{n}\times\mathbb{R}^{n})$
and $S_{\mu\nu}^{m_{1},m_{2}}(\mathbb{R}^{n}\times\mathbb{R}^{n})$,
respectively.
\end{defn}
It is clear that we could do the same definitions also for $\left(x,\xi\right)\in\mathbb{R}^{n_{1}}\times\mathbb{R}^{n_{2}}$.
We denote theses spaces by $S^{m_{1},m_{2}}(\mathbb{R}^{n_{1}}\times\mathbb{R}^{n_{2}})$,
$S_{\mu\nu}^{m_{1},m_{2}}(\mathbb{R}^{n_{1}}\times\mathbb{R}^{n_{2}})$
and so on.

We are mostly interested in elliptic symbols.
\begin{defn}
\label{SG ellipticity}Let $a\in S^{m_{1},m_{2}}(\mathbb{R}^{n}\times\mathbb{R}^{n},\mathcal{B}(\mathbb{C}^{p},\mathbb{C}^{q}))$,
$\left(m_{1},m_{2}\right)\in\mathbb{R}^{2}$. We say that the symbol
$a$ is left (right) elliptic if there are constants $R>0$ and $C>0$
such that if $\left|\left(x,\xi\right)\right|\ge R$, then $a(x,\xi)$
has a left (right) inverse $(x,\xi)\mapsto b(x,\xi)$ such that $\left\Vert b(x,\xi)\right\Vert _{\mathcal{B}\left(\mathbb{C}^{q},\mathbb{C}^{p}\right)}\le C\left\langle x\right\rangle ^{-m_{2}}\left\langle \xi\right\rangle ^{-m_{1}}$.
In particular, $q\ge p$ ($q\le p$). A symbol is elliptic iff it
is left and right elliptic. In this case $p=q$ and $\left\Vert a(x,\xi)^{-1}\right\Vert _{\mathcal{B}\left(\mathbb{C}^{q},\mathbb{C}^{p}\right)}\le C\left\langle x\right\rangle ^{-m_{2}}\left\langle \xi\right\rangle ^{-m_{1}}$,
if $\left|\left(x,\xi\right)\right|\ge R$.

The same ellipticity definition holds for $a\in S_{\mu\nu}^{m_{1},m_{2}}(\mathbb{R}^{n}\times\mathbb{R}^{n},\mathcal{B}(\mathbb{C}^{p},\mathbb{C}^{q}))$.
\end{defn}
For a symbol $a\in S^{m_{1},m_{2}}(\mathbb{R}^{n}\times\mathbb{R}^{n},\mathcal{B}(\mathbb{C}^{p},\mathbb{C}^{q}))$,
$\left(m_{1},m_{2}\right)\in\mathbb{R}^{2}$, left ellipticity is
equivalent to the ellipticity of the symbol $a^{*}a\in S^{2m_{1},2m_{2}}(\mathbb{R}^{n}\times\mathbb{R}^{n},\mathcal{B}(\mathbb{C}^{p},\mathbb{C}^{p}))$.
Hence, if $b$ is a parametrix of $a^{*}a$, then $ba^{*}$ is a left
parametrix of $a$. The analogous result holds for right elliptic
symbols.

As usual we define pseudo-differential and regularizing operators
associated with these symbols.
\begin{defn}
For each symbol $a$ in the classes $S^{m_{1},m_{2}}(\mathbb{R}^{n}\times\mathbb{R}^{n})$
and $S_{\mu\nu}^{m_{1},m_{2}}(\mathbb{R}^{n}\times\mathbb{R}^{n})$,
we define a pseudo-differential operator $A=op(a):\mathcal{S}(\mathbb{R}^{n})\to\mathcal{S}(\mathbb{R}^{n})$
by the formula: 
\[
Au(x)=\frac{1}{\left(2\pi\right)^{n}}\int e^{ix\xi}a(x,\xi)\hat{u}(\xi)d\xi,
\]
where $\hat{u}(\xi)=\int e^{-ix\xi}u(x)dx.$ For matrix symbols $a=(a_{ij})$
that belong to $S^{m_{1},m_{2}}(\mathbb{R}^{n}\times\mathbb{R}^{n},\mathcal{B}(\mathbb{C}^{p},\mathbb{C}^{q}))$
and $S_{\mu\nu}^{m_{1},m_{2}}(\mathbb{R}^{n}\times\mathbb{R}^{n},\mathcal{B}(\mathbb{C}^{p},\mathbb{C}^{q}))$,
we define $A=op(a):\mathcal{S}(\mathbb{R}^{n})^{\oplus p}\to\mathcal{S}(\mathbb{R}^{n})^{\oplus q}$
by: 
\begin{equation}
\left(Au\right)_{k}(x)=\sum_{j=1}^{p}\frac{1}{(2\pi)^n}\int e^{ix\xi}a_{kj}(x,\xi)\hat{u}_{j}(\xi)d\xi.\label{eq:def pseudo}
\end{equation}

\end{defn}

\begin{defn}
Let $\theta>1$. A linear continuous operator from $\mathcal{S}{}_{\theta}^{\theta}(\mathbb{R}^{n})^{\oplus p}$
to $\mathcal{S}{}_{\theta}^{\theta}(\mathbb{R}^{n})^{\oplus q}$ is
said to be $\theta-$regularizing operator if it extends to a linear
continuous map from $\mathcal{S}{}_{\theta}^{\theta'}(\mathbb{R}^{n})^{\oplus p}$
to $\mathcal{S}{}_{\theta}^{\theta}(\mathbb{R}^{n})^{\oplus q}$.
\end{defn}
Operators whose kernel is a $q\times p$ matrix with entries in $\mathcal{S}_{\theta}^{\theta}(\mathbb{R}^{n}\times\mathbb{R}^{n})$,
also denoted by $\mathcal{S}_{\theta}^{\theta}(\mathbb{R}^{n}\times\mathbb{R}^{n},\mathcal{B}(\mathbb{C}^{p},\mathbb{C}^{q}))$,
are $\theta-$regularizing operators. 

Let us now give the properties of the pseudo-differential operators
with symbols in $S_{\mu\nu}^{m_{1},m_{2}}(\mathbb{R}^{n}\times\mathbb{R}^{n},\mathcal{B}(\mathbb{C}^{p},\mathbb{C}^{q}))$,
which we shall need. They can be found in \cite[Chapter 6]{Rodinonicola}.
\begin{prop}
\label{pro:continuity in gelfandshilov}Let $a\in S_{\mu\nu}^{m_{1},m_{2}}(\mathbb{R}^{n}\times\mathbb{R}^{n},\mathcal{B}(\mathbb{C}^{p},\mathbb{C}^{q}))$,
$\left(m_{1},m_{2}\right)\in\mathbb{R}^{2}$, $\mu$ and $\nu$ be
real numbers such that $\mu\ge1$, $\nu\ge1$, $p$ and $q$ integers
such that $p\ge1$ and $q\ge1$. For every $\theta\ge\max\left\{ \mu,\nu\right\} $,
the operator defined in \ref{eq:def pseudo} is a linear continuous
operator from $\mathcal{S}_{\theta}^{\theta}(\mathbb{R}^{n})^{\oplus p}$
to $\mathcal{S}_{\theta}^{\theta}(\mathbb{R}^{n})^{\oplus q}$.

Let us define $H^{s_{1},s_{2}}(\mathbb{R}^{n}):=\left\{ u\in\mathcal{S}'(\mathbb{R}^{n});\,\left\langle x\right\rangle ^{s_{2}}\left\langle D\right\rangle ^{s_{1}}u\in L^{2}(\mathbb{R}^{n})\right\} $.
Then the above operator extends to a continuous operator from $H^{s_{1},s_{2}}(\mathbb{R}^{n})^{\oplus p}$
to $H^{s_{1}-m_{1},s_{2}-m_{2}}(\mathbb{R}^{n})^{\oplus q}$ \end{prop}
\begin{defn}
\label{def:Formal sums}Let us denote by $FS_{\mu\nu}^{m_{1},m_{2}}(\mathbb{R}^{n}\times\mathbb{R}^{n})$
the space of all formal sums $\sum_{j\ge0}a_{j}$, where the functions
$a_{j}\in C^{\infty}(\mathbb{R}^{n}\times\mathbb{R}^{n})$ satisfy
the following condition: There are constants $B$, $C$ and $D>0$
such that
\[
\left|\partial_{\xi}^{\alpha}\partial_{x}^{\beta}a_{j}(x,\xi)\right|\le CD^{\left|\alpha\right|+\left|\beta\right|+2j}\left(\alpha!\right)^{\mu}\left(\beta!\right)^{\nu}\left(j!\right)^{\mu+\nu-1}\left\langle x\right\rangle ^{m_{2}-j-\left|\beta\right|}\left\langle \xi\right\rangle ^{m_{1}-j-\left|\alpha\right|},
\]
whenever $\left\langle x\right\rangle \ge Bj^{\mu+\nu-1}$ or $\left\langle \xi\right\rangle \ge Bj^{\mu+\nu-1}$.
In the same way, by $FS_{\mu\nu}^{m_{1},m_{2}}(\mathbb{R}^{n}\times\mathbb{R}^{n},\mathcal{B}(\mathbb{C}^{p},\mathbb{C}^{q}))$,
we denote the space of all formal sums $\sum_{j\ge0}a_{j}$, where
the functions $a_{j}\in C^{\infty}(\mathbb{R}^{n}\times\mathbb{R}^{n},\mathcal{B}(\mathbb{C}^{p},\mathbb{C}^{q}))$
are such that the formal sum of each of its entry belongs to $FS_{\mu\nu}^{m_{1},m_{2}}(\mathbb{R}^{n}\times\mathbb{R}^{n})$.
\end{defn}

\begin{defn}
\label{def:Asymptotic expansion}We say that $a\in S_{\mu\nu}^{m_{1},m_{2}}(\mathbb{R}^{n}\times\mathbb{R}^{n})$
has the asymptotic expansion $a\sim\sum_{j\ge0}a_{j}$, if $\sum_{j\ge0}a_{j}\in FS_{\mu\nu}^{m_{1},m_{2}}(\mathbb{R}^{n}\times\mathbb{R}^{n})$
and if there exist constants $B$, $C$ and $D>0$ such that
\[
\left|\partial_{\xi}^{\alpha}\partial_{x}^{\beta}\left(a-\sum_{j=0}^{N-1}a_{j}\right)(x,\xi)\right|\le CD^{\left|\alpha\right|+\left|\beta\right|+2N}\left(\alpha!\right)^{\mu}\left(\beta!\right)^{\nu}\left(N!\right)^{\mu+\nu-1}\left\langle x\right\rangle ^{m_{2}-N-\left|\beta\right|}\left\langle \xi\right\rangle ^{m_{1}-N-\left|\alpha\right|},
\]
whenever $\left\langle x\right\rangle \ge BN^{\mu+\nu-1}$ or $\left\langle \xi\right\rangle \ge BN^{\mu+\nu-1}$.
The analogous holds for matricial symbols.\end{defn}
\begin{prop}
Let $\sum_{j\ge0}a_{j}\in FS_{\mu\nu}^{m_{1},m_{2}}(\mathbb{R}^{n}\times\mathbb{R}^{n},\mathcal{B}(\mathbb{C}^{p},\mathbb{C}^{q}))$.
Then there exists a symbol $a\in S_{\mu\nu}^{m_{1},m_{2}}(\mathbb{R}^{n}\times\mathbb{R}^{n},\mathcal{B}(\mathbb{C}^{p},\mathbb{C}^{q}))$
such that $a\sim\sum_{j\ge0}a_{j}$. Moreover if $a$ and $b$ are
two functions that belong to $S_{\mu\nu}^{m_{1},m_{2}}(\mathbb{R}^{n}\times\mathbb{R}^{n},\mathcal{B}(\mathbb{C}^{p},\mathbb{C}^{q}))$
and are such that $a\sim\sum_{j\ge0}a_{j}$ and $b\sim\sum_{j\ge0}a_{j}$,
then $a-b\in\mathcal{S}_{\theta}^{\theta}(\mathbb{R}^{n}\times\mathbb{R}^{n},\mathcal{B}(\mathbb{C}^{p},\mathbb{C}^{q}))$,
that is, $op(a)-op(b)$ is a $\theta$-regularizing operator for any
$\theta\ge\mu+\nu-1$. 
\end{prop}

\begin{prop}
\label{prop:composicao e adjuntos}Let $\mu>1$ and $\nu>1$. Let
us consider the symbols $a\in S_{\mu\nu}^{m_{1},m_{2}}(\mathbb{R}^{n}\times\mathbb{R}^{n},\mathcal{B}(\mathbb{C}^{p},\mathbb{C}^{q}))$,
$\left(m_{1},m_{2}\right)\in\mathbb{R}^{2}$, and $b\in S_{\mu\nu}^{m_{1}',m_{2}'}(\mathbb{R}^{n}\times\mathbb{R}^{n},\mathcal{B}(\mathbb{C}^{r},\mathbb{C}^{p}))$,
$\left(m_{1}',m_{2}'\right)\in\mathbb{R}^{2}$. Then, for $\theta\ge\mu+\nu-1$:

(i) There is a symbol $c\in S_{\mu\nu}^{m_{1}+m_{1}',m_{2}+m_{2}'}(\mathbb{R}^{n}\times\mathbb{R}^{n},\mathcal{B}(\mathbb{C}^{r},\mathbb{C}^{q}))$,
also denoted by $a\sharp b$, and a $\theta-$regularizing operator
$R:\mathcal{S}_{\theta}^{\theta'}(\mathbb{R}^{n})^{\oplus r}\to\mathcal{S}_{\theta}^{\theta}(\mathbb{R}^{n})^{\oplus q}$,
where $\theta\ge\mu+\nu-1$, such that $op(c)=op(a)\circ op(b)+R$.
Its symbol has the following asymptotic expansion
\[
a\#b\sim\sum_{\alpha}\frac{1}{\alpha!}\partial_{\xi}^{\alpha}aD_{x}^{\alpha}b.
\]

(ii) There is a symbol $c\in S^{m_{1},m_{2}}(\mathbb{R}^{n}\times\mathbb{R}^{n},\mathcal{B}(\mathbb{C}^{q},\mathbb{C}^{p}))$,
also denoted $a^{*}$, and a $\theta-$regularizing operator $R:\mathcal{S}_{\theta}^{\theta'}(\mathbb{R}^{n})^{\oplus q}\to\mathcal{S}_{\theta}^{\theta}(\mathbb{R}^{n})^{\oplus p}$,
where $\theta\ge\mu+\nu-1$, such that, for all $u$, $v$ $\in\mathcal{S}(\mathbb{R}^{n})$,
the following equality holds
\[
\left(op(a)u,v\right)_{L^{2}(\mathbb{R}^{n})^{\oplus q}}=\left(u,op(a^{*})v\right)_{L^{2}(\mathbb{R}^{n})^{\oplus p}}+\left(u,Rv\right)_{L^{2}(\mathbb{R}^{n})^{\oplus p}}.
\]

The symbol $a^{*}$ has the following asymptotic expansion
\[
a^{*}\sim\sum_{\alpha}\frac{1}{\alpha!}\partial_{\xi}^{\alpha}D_{x}^{\alpha}\overline{a}.
\]

\end{prop}
We finally state the following important regularity result:
\begin{thm}
\label{thm:Regularity} Let $a\in S_{\mu\nu}^{m_{1},m_{2}}(\mathbb{R}^{n}\times\mathbb{R}^{n},\mathcal{B}(\mathbb{C}^{p},\mathbb{C}^{q}))$,
$\mu>1$, $\nu>1$, be a left (right) elliptic symbol. Then there
is a symbol called left (right) parametrix $b\in S_{\mu\nu}^{-m_{1},-m_{2}}(\mathbb{R}^{n}\times\mathbb{R}^{n},\mathcal{B}(\mathbb{C}^{q},\mathbb{C}^{p}))$
such that $op(b)op(a)=I+R$ $(op(a)op(b)=I+R)$, where $R:\mathcal{S}_{\theta}^{\theta'}(\mathbb{R}^{n})^{\oplus p}\to\mathcal{S}_{\theta}^{\theta}(\mathbb{R}^{n})^{\oplus p}$
$(R:\mathcal{S}_{\theta}^{\theta'}(\mathbb{R}^{n})^{\oplus q}\to\mathcal{S}_{\theta}^{\theta}(\mathbb{R}^{n})^{\oplus q})$
is a $\theta$-regularizing operator and $\theta$ is any number that
satisfies $\theta\ge\mu+\nu-1$.

In particular, if $a$ is a left elliptic symbol, $u\in\mathcal{S}_{\theta}^{\theta'}(\mathbb{R}^{n})^{\oplus p}$
and $op(a)u\in\mathcal{S}_{\theta}^{\theta}(\mathbb{R}^{n})^{\oplus q}$
for $\theta\ge\mu+\nu-1$, then $u\in\mathcal{S}_{\theta}^{\theta}(\mathbb{R}^{n})^{\oplus p}$.
\end{thm}
Note that if $a$ is a left elliptic symbol of a differential operator,
then $\mu=1$. Hence if $u\in\mathcal{S}_{\theta}^{\theta'}(\mathbb{R}^{n})^{\oplus p}$
is such that $op(a)u\in\mathcal{S}_{\theta}^{\theta}(\mathbb{R}^{n})^{\oplus q}$
for $\theta>\nu$, then $u\in\mathcal{S}_{\theta}^{\theta}(\mathbb{R}^{n})^{\oplus p}$.

\section{Regularity Results\label{sec:Regularity-Results}}

\subsection{Elliptic SG boundary value problems on the complement of compact
sets.}

In this section we prove regularity in Gelfand-Shilov spaces of solutions
of SG boundary value problems on the complement of compact sets, as
introduced by C. Parenti \cite[Section 3]{Parenti}.

Let $U$ be a bounded open set such that its boundary $\Gamma=\partial U$
is a Gevrey $(n-1)$-manifold of order $\Theta>1$, $U$ being locally
on one side of $\Gamma$. Let $\Omega=\mathbb{R}^{n}\backslash\overline{U}$.
We consider the following boundary value problem:

\[
\begin{array}{c}
Pu=f,\,\,\,\mbox{in}\,\Omega\\
B^{j}u=g_{j},\,\,\,\mbox{on}\,\Gamma,\, j=1,2,...,r
\end{array},
\]
where:

\vspace{0,2 cm}

a) $P(x,D)=\sum_{|\alpha|\le m_{1}}a_{\alpha}(x)D_{x}^{\alpha}$ is
a differential operator on $\mathbb{R}^{n}$ and $m_{1}=2r$. We assume
that the functions $a_{\alpha}\in C^{\infty}(\mathbb{R}^{n})$ satisfy
the following estimates for some $\nu\ge1$
\[
\left|\partial_{x}^{\beta}a_{\alpha}(x)\right|\le CD^{\left|\beta\right|}\left(\beta!\right)^{\nu}\left\langle x\right\rangle ^{m_{2}-\left|\beta\right|},\,\forall x\in\mathbb{R}^{n}.
\]

Hence the function $a\in C^{\infty}(\mathbb{R}^{n}\times\mathbb{R}^{n})$
given by $a\left(x,\xi\right)=\sum_{|\alpha|\le m_{1}}a_{\alpha}(x)\xi^{\alpha}$
belongs to $S_{1\nu}^{m_{1},m_{2}}(\mathbb{R}^{n}\times\mathbb{R}^{n})$

\vspace{0,2 cm}

b) For each $j=1,...,r$, we associate an integer number $0\le m_{1j}\le m_{1}-1$
and $m_{2j}\in\mathbb{R}$. Let $B=\left(B_{j,k}\right),\,\,\mbox{with}\,\,\,1\le j\le r\,\,\,\mbox{and}\,\,\,0\le k\le m_{1}-1,$
be a matrix, where $B_{j,k}$ is a differential operator of order
$m_{1j}-k$ on $\Gamma$ whose coefficients are Gevrey functions of
order $\Theta$. We assume also that $B_{j,k}=0$, if $k>m_{1j}$.
For each $u\in\mathcal{S}(\Omega)$, we define $\gamma\left(u\right):=\left(\gamma_{0}\left(u\right),...,\gamma_{m_{1}-1}\left(u\right)\right)$,
where $\gamma_{j}$ is defined using the charts of Definition \ref{def:Gevrey submanifolds}:
$\gamma_{j}\left(u\right)\circ\psi^{-1}(x):=\lim_{x_{n}\to0^{+}}D_{x_{n}}^{j}\left(u\circ\psi^{-1}\right)(x,x_{n})$.
The derivative $D_{x_{n}}$ is the one associated with the field $\nu$,
again as in Definition \ref{def:Gevrey submanifolds}. The operator
$B^{j}:\mathcal{S}(\Omega)\to C^{\infty}(\Gamma)^{\oplus r}$ is defined
as $B^{j}u=\sum_{k=0}^{m_{1}-1}B_{j,k}(x',D')\gamma_{k}\left(u\right),\,\,\, j=1,2,...,r$.

\vspace{0,2 cm}

c) The symbol $a$ is SG-elliptic, properly elliptic in the classical
sense and the boundary value problem satisfies the usual Lopatinski-Shapiro
condition at the boundary. We recall the definition below.

\vspace{0,2 cm}

d) There is a $\theta>\nu$ and $\theta\ge\Theta$ such that the functions
$g_{j}$ are Gevrey functions of order $\theta$ in $\Gamma$ and
$f\in\mathcal{S}_{\theta}^{\theta}(\Omega)$.
\begin{defn}
1) Let us define $a_{(m_{1})}\left(x,\xi\right)=\sum_{|\alpha|=m_{1}}a_{\alpha}(x)\xi^{\alpha}$.
We say that the function $a$ is properly elliptic in the classical
sense if it is elliptic - it is non zero if $\xi\ne0$ - and for all
$x\in\Gamma$, $\xi_{1}$ and $\xi_{2}$ linearly independent vectors
in $\mathbb{R}^{n}$, the polynomial $z\in\mathbb{C}\mapsto a_{(m_{1})}\left(x,\xi_{1}+z\xi_{2}\right)$
has exactly $r=\frac{m_{1}}{2}$ roots with positive imaginary part
- and, hence, $r$ roots with negative imaginary part. We denote these
roots by $\tau_{1}(x,\xi_{1},\xi_{2})$, ..., $\tau_{r}(x,\xi_{1},\xi_{2})$
and we set $a_{(m_{1})}^{+}(x,\xi_{1},\xi_{2})\left(z\right):=\prod_{j=1}^{r}\left(z-\tau_{j}(x,\xi_{1},\xi_{2})\right)$.

2) Let us write $B^{j}u(x)=\sum_{\left|\alpha\right|\le m_{1j}}b_{\alpha}^{j}(x)D^{\alpha}$,
for $x\in\Gamma$. We define the polynomials 
\[
b_{(m_{1j})}^{j}(x,\xi,\xi')(z):=\sum_{\left|\alpha\right|=m_{1j}}b_{\alpha}^{j}(x)\left(\xi+z\xi'\right)^{\alpha},
\]
where $x\in\Gamma$, $\xi$ is tangent to $\Gamma$ and $\xi'$ is
normal to $\Gamma$. The boundary value problem satisfies the classical
Lopatinski-Shapiro (or covering) condition if, for all $x\in\Gamma$,
$\xi\ne0$ tangent to $\Gamma$ and $\xi'\ne0$ normal to $\Gamma$,
the polynomials $z\in\mathbb{C}\mapsto b_{(m_{1j})}^{j}(x,\xi,\xi')(z)$
are linearly independent modulo $z\in\mathbb{C}\mapsto a_{(m_{1})}^{+}(x,\xi,\xi')\left(z\right)$
.
\end{defn}
Let us recall the classical result about Gevrey regularity of elliptic
boundary value problems:
\begin{thm}
\label{thm:lions magenes}(Theorem 1.3 of J. Lions and E. Magenes
\cite{lionsmagenes3}) Let $0<\rho_{0}<1$ be a fixed constant and
let us consider the following boundary value problem in $B_{\rho_{0}}(0)\cap\mathbb{R}_{+}^{n}$:
\[
\begin{array}{c}
Pu=f,\,\,\,\mbox{in}\, B_{\rho_{0}}(0)\cap\mathbb{R}_{+}^{n}\\
B^{j}u=g_{j},\,\,\,\mbox{on}\, B_{\rho_{0}}(0)\cap\partial\left(\mathbb{R}_{+}^{n}\right),\, j=1,2,...,r
\end{array},
\]
where:

1) $P(x,D)u=\sum_{\left|\alpha\right|\le m_{1}}a_{\alpha}(x)D^{\alpha}u(x)$,
$m_{1}=2r$, is a properly elliptic operator on $B_{\rho_{0}}(0)\cap\partial\left(\mathbb{R}_{+}^{n}\right)$
and $a_{\alpha}$ are restrictions of Gevrey functions defined on
$\mathbb{R}^{n}$ of order $\beta>1$ to $B_{\rho_{0}}(0)\cap\mathbb{R}_{+}^{n}$.

2) $B^{j}u(x')=\sum_{\alpha\le m_{1j}}b_{j\alpha}(x')D^{\alpha}u(x',0)$
are $r$ boundary operators and $b_{j\alpha}$ are Gevrey functions
of order $\beta$ defined on $B_{\rho_{0}}(0)\cap\partial\left(\mathbb{R}_{+}^{n}\right)$.

If $u\in C^{\infty}\left(\overline{B_{\rho_{0}}(0)\cap\mathbb{R}_{+}^{n}}\right)$,
$f$ and $g_{j}$, $1\le j\le r$, are Gevrey functions of order $\beta$
and the above boundary value problem satisfies the classical Lopatinski-Shapiro
condition, then there is a $\rho'<\rho_{0}$, such that $\left.u\right|_{B_{\rho'}(0)\cap\mathbb{R}_{+}^{n}}$
is the restriction of a Gevrey function of order $\beta$ defined
on $\mathbb{R}^{n}$ to $B_{\rho'}(0)\cap\mathbb{R}_{+}^{n}$.
\end{thm}
We finally prove our main result on the complement of compact sets:
\begin{thm}
\label{thm:Teorema principal no complemento de abertos}(Main Theorem
on the complement of compact sets) Let $u\in H^{s_{1},s_{2}}(\Omega)$,
the space of restrictions of distributions in $H^{s_{1},s_{2}}(\mathbb{R}^{n})$
to $\Omega$, $s_{1}\ge m_{1}$, be a solution of 
\[
\begin{array}{c}
Pu=f,\,\,\,\mbox{in}\,\,\Omega\\
B^{j}u=g_{j},\,\,\,\mbox{on}\,\,\Gamma,\, j=1,2,...,r
\end{array},
\]
where the boundary value problem satisfies the conditions $a$, $b$,
$c$ and $d$. Then $u\in\mathcal{S}_{\theta}^{\theta}(\Omega)$.\end{thm}
\begin{proof}
As $f$ belongs to $\mathcal{S}(\Omega)=\cap_{\left(s_{1},s_{2}\right)\in\mathbb{R}^{2}}H^{s_{1},s_{2}}(\Omega)$
and $g_{j}$ belongs to $C^{\infty}(\Gamma)=\cap_{s\in\mathbb{R}}H^{s}(\Gamma)$,
for all $j$, we conclude that $u\in\mathcal{S}(\Omega)$, according
to C. Parenti \cite[Section 3]{Parenti}. For each $x\in\Gamma$,
there exists a neighborhood $\mathcal{O}_{x}\subset\mathbb{R}^{n}$
of $x$, $r_{x}>0$ and a Gevrey diffeomorphism $\psi_{x}:\mathcal{O}_{x}\to B_{r_{x}}(0)$
of order $\Theta$, as in Definition \ref{def:Gevrey submanifolds}.
Due to Theorem \ref{thm:lions magenes}, there is a ball $B_{\rho_{x}}(0)\subset B_{r_{x}}(0)$,
$\rho_{x}<1$, such that $u\circ\psi_{x}^{-1}:B_{\rho_{x}}(0)\cap\mathbb{R}_{+}^{n}\to\mathbb{C}$
is the restriction of a Gevrey function $\tilde{u}_{x}:\mathbb{R}^{n}\to\mathbb{C}$
of order $\theta$. Let us choose $x_{1}$, ..., $x_{N}$ such that
$\Gamma\subset\cup_{j=1}^{N}\psi_{x_{j}}^{-1}\left(B_{\rho_{x_{j}}}(0)\right)$.
Let us also choose Gevrey functions with compact support $\chi_{1}$,
..., $\chi_{N}$ of order $\theta$ taking values on $[0,1]$ and
such that $\sum_{j=1}^{N}\chi_{j}=1$ in a neighborhood of $\Gamma$
and that $\mbox{supp}(\chi_{j})\subset\psi_{x_{j}}^{-1}\left(B_{\rho_{x_{j}}}(0)\right)$.
Let us define
\[
\tilde{u}(x)=\sum_{j=1}^{N}\chi_{j}(x)\tilde{u}_{x_{j}}\circ\psi_{x_{j}}(x).
\]

This is a Gevrey function of order $\theta$. Moreover, there is a
bounded neighborhood of $\Gamma$, $V\subset\mathbb{R}^{n}$, such
that $\left.\tilde{u}\right|_{V\cap\Omega}=\left.u\right|_{V\cap\Omega}$.
Let us now choose a Gevrey function of order $\theta$, $\chi:\mathbb{R}^{n}\to\mathbb{C}$,
such that $\chi(x)=1$ in a neighborhood of the complement of $V\cup\Omega$
and $\chi(x)=0$ in a neighborhood of $\overline{\Omega}$. Hence,
on $\Omega$, we have $P\left(\chi u\right)=f-P\left(\left(1-\chi\right)\tilde{u}\right)$.
However $f-P\left(\left(1-\chi\right)\tilde{u}\right)$ is zero in
a neighborhood of $\Gamma$ and $P\left(\left(1-\chi\right)\tilde{u}\right)$
is a Gevrey function of order $\theta$ with compact support. This
means that $f-P\left(\left(1-\chi\right)\tilde{u}\right)$ can be
extended to a Gelfand-Shilov function in $\mathcal{S}_{\theta}^{\theta}(\mathbb{R}^{n})$
- we only have to extend by zero on $\overline{U}$. As $\theta>\nu$,
we conclude, using Theorem \ref{thm:Regularity}, that $\chi u\in\mathcal{S}_{\theta}^{\theta}(\mathbb{R}^{n})$.
As $u$ is the restriction of the function $\chi u+\left(1-\chi\right)\tilde{u}$
to $\Omega$, we conclude that $u\in\mathcal{S}_{\theta}^{\theta}(\Omega)$.
\end{proof}

\subsection{Elliptic SG boundary value problems on the half-space.}

In this section we prove regularity in Gelfand-Shilov spaces of solutions
of SG boundary value problems on the half-space, as introduced by
H. O. Cordes and A. K. Erkip \cite{Cordeserkip,Erkipcommunications,Erkiplecturenotes}.

We consider the following boundary value problem:

\[
\begin{array}{c}
Pu=f,\,\,\,\mbox{in}\,\,\mathbb{R}_{+}^{n}\\
B^{j}u=g_{j},\,\,\,\mbox{on}\,\,\mathbb{R}^{n-1},\, j=1,2,...,r
\end{array},
\]
where:

\vspace{0,2 cm}

a) $P(x,D)=\sum_{|\alpha|\le m_{1}}a_{\alpha}(x)D_{x}^{\alpha}$ is
a differential operator on $\mathbb{R}^{n}$ and $m_{1}=2r$. We assume
that the functions $a_{\alpha}\in C^{\infty}(\mathbb{R}^{n})$ satisfy
the following estimates for $\nu\ge1$
\[
\left|\partial_{x}^{\beta}a_{\alpha}(x)\right|\le CD^{\left|\beta\right|}\left(\beta!\right)^{\nu}\left\langle x\right\rangle ^{m_{2}-\left|\beta\right|}.
\]

Hence the function $a\in C^{\infty}(\mathbb{R}^{n}\times\mathbb{R}^{n})$
given by $a\left(x,\xi\right)=\sum_{|\alpha|\le m_{1}}a_{\alpha}(x)\xi^{\alpha}$
belongs to $S_{1\nu}^{m_{1},m_{2}}(\mathbb{R}^{n}\times\mathbb{R}^{n})$

\vspace{0,2 cm}

b) For each $j=1,...,r$, we associate an integer $0\le m_{1j}\le m_{1}-1$
and $m_{2j}\in\mathbb{R}$. Let $B=\left(B_{j,k}\right),\,\,\mbox{with}\,\,\,1\le j\le r\,\,\,\mbox{and}\,\,\,0\le k\le m_{1}-1$
be a matrix, where $B_{j,k}$ is a pseudo-differential operator, whose
symbol belongs to $S_{1\nu}^{m_{1j}-k,m_{2j}}\left(\mathbb{R}^{n-1}\times\mathbb{R}^{n-1}\right)$,
$j=1,...,r$. We assume also that $B_{j,k}=0$, if $k>m_{1j}$. For
each $u\in\mathcal{S}(\mathbb{R}_{+}^{n})$, we define $\gamma\left(u\right):=\left(\gamma_{0}\left(u\right),...,\gamma_{m_{1}-1}\left(u\right)\right)$,
where $\gamma_{j}\left(u\right):=\lim_{x_{n}\to0^{+}}\left(D_{x_{n}}^{j}u\right)(x,x_{n})$.
The operator $B^{j}:\mathcal{S}(\mathbb{R}_{+}^{n})\to\mathcal{S}(\mathbb{R}^{n-1})^{\oplus r}$
is defined as 
\[
B^{j}u=\sum_{k=0}^{m_{1}-1}B_{j,k}(x',D')\gamma_{k}\left(u\right),\,\,\, j=1,2,...,r.
\]

\vspace{0,2 cm}

c) The boundary value problem is SG elliptic, as defined below.

\vspace{0,2 cm}

d) There is a $\theta>\nu$ such that $g_{j}\in\mathcal{S}_{\theta}^{\theta}(\mathbb{R}^{n-1})$,
$\forall j$, and $f\in\mathcal{S}_{\theta}^{\theta}(\mathbb{R}_{+}^{n})$.
\begin{defn}
We say that the above boundary value problem is SG elliptic if it
satisfies the following conditions \cite{Erkipcommunications,Erkiplecturenotes}:

1) Let us define the function $a_{(x',\xi')}(z):=\left\langle x'\right\rangle ^{-m_{2}}\left\langle \xi'\right\rangle ^{-m_{1}}a\left(x',0,\xi',\left\langle \xi'\right\rangle z\right)$.
The function $a$ is $SG$-properly elliptic: it is $SG$-elliptic
as in Definition \ref{SG ellipticity} and there is an $R>0$ such
that, for $\left|(x,\xi)\right|\ge R$, the polynomial $z\in\mathbb{C}\mapsto a_{(x',\xi')}(z)$
has exactly $r$ roots with positive imaginary part - and $r$ roots
with negative imaginary part. We denote these roots by $\tau_{1}(x',\xi')$,
..., $\tau_{r}(x',\xi')$ and we set $a_{(x',\xi')}^{+}(z):=\prod_{j=1}^{r}\left(z-\tau_{j}(x',\xi')\right)$.

2) Let us define the polynomials $b_{(x',\xi')}^{j}(z):=\sum_{k=0}^{m_{1}-1}B_{j,k}(x',\xi')\left\langle x'\right\rangle ^{-m_{2j}}\left\langle \xi'\right\rangle ^{k-m_{1j}}z^{k}$.
The boundary value problem satisfies the SG-Lopatinski-Shapiro (or
covering) condition: there exists $R>0$ such that if $\left|(x',\xi')\right|\ge R$,
then the polynomials $b_{(x',\xi')}^{j}(z)$ are uniformly and linearly
independent modulo $a_{(x',\xi')}^{+}(z)$. This means that $b_{(x',\xi')}^{j}(z)=\tilde{b}_{(x',\xi')}^{j}(z)\,\mbox{mod}\, a_{(x',\xi')}^{+}(z)$,
where $\tilde{b}_{(x',\xi')}^{j}(z)=\sum_{k=0}^{r-1}\tilde{b}_{(x',\xi')}^{j,k}z^{k}$,
and for $\left|(x',\xi')\right|\ge R$, there exists a constant $C>0$,
independent of $(x',\xi')$, such that $\left|\det\left(\tilde{b}_{(x',\xi')}^{j,k}(z)\right)\right|\ge C$.\end{defn}
\begin{example}
Let us consider the Dirichlet problem:
\[
\begin{array}{c}
Pu=f,\,\,\,\mbox{in}\,\,\mathbb{R}_{+}^{n}\\
\left\langle x'\right\rangle ^{m_{2j}}\gamma^{j-1}\left(u\right)=g_{j},\,\,\,\mbox{on}\,\,\mathbb{R}^{n-1},\, j=1,2,...,r
\end{array},
\]
where $P(x,D)=\sum_{\left|\alpha\right|\le m_{1}}a_{\alpha}(x)D^{\alpha}$,
$m_{1}=2r$, is a $SG$-properly elliptic differential operator. Then
$b_{(x',\xi')}^{j}(z)=\tilde{b}_{(x',\xi')}^{j}(z)=z^{j}$ and $\tilde{b}_{(x',\xi')}^{j,k}=\delta_{jk}$.
This clearly satisfies the SG-Lopatinski-Shapiro condition.\end{example}
\begin{rem}
It is very important to note, as in \cite{Erkipcommunications,Erkiplecturenotes},
that, if $a\in S^{m_{1},m_{2}}\left(\mathbb{R}^{n}\times\mathbb{R}^{n}\right)$
is an SG elliptic symbol of a differential operator, then there is
a constant $D>0$ that does not depend on $\left(x,\xi'\right)$,
such that the roots of the polynomial $z\in\mathbb{C}\mapsto a\left(x,\xi',z\right)$
satisfy $\left|z\right|\le D\left\langle \xi'\right\rangle $, for
all $\left(x,\xi'\right)$. If we write $a\left(x,\xi',z\right)=\sum_{j=0}^{m_{1}}P_{j}(x,\xi')z^{j}$,
then $P_{j}\in S^{m_{1}-j,m_{2}}\left(\mathbb{R}^{n}\times\mathbb{R}^{n-1}\right)$
and $P_{m_{1}}$ only depends on $x$. By the ellipticity assumption,
$\left|P_{m_{1}}(x)\right|\ge C\left\langle x\right\rangle ^{m_{1}}$
for a constant $C>0$ that does not depend on $x$. The result follows
then easily from the simple fact that the roots of a polynomial $P(z)=\sum_{j=0}^{N}P_{j}z^{j}$
belong to the ball of radius $\max_{j}\left\{ \left(N\left|\frac{P_{j}}{P_{N}}\right|\right)^{\frac{1}{N-j}}\right\} $. 

The Theorem below is our main result on the half-space.\end{rem}
\begin{thm}
\label{thm:teorema principal semi plano} (Main Theorem on the half-space)
Let $u\in H^{s_{1},s_{2}}(\mathbb{R}_{+}^{n})$, the space of restrictions
of distributions in $H^{s_{1},s_{2}}(\mathbb{R}^{n})$ to $\mathbb{R}_{+}^{n}$,
$s_{1}\ge m_{1}$, be a solution of 
\begin{equation}
\begin{array}{c}
Pu=f,\,\,\,\mbox{in}\,\,\mathbb{R}_{+}^{n}\\
B^{j}u=g_{j},\,\,\,\mbox{on}\,\,\mathbb{R}^{n-1},\, j=1,2,...,r
\end{array},\label{eq:Equacao principal no semi plano}
\end{equation}
where the boundary value problem satisfies the conditions $a$, $b$,
$c$ and $d$. Then $u\in\mathcal{S}_{\theta}^{\theta}(\mathbb{R}_{+}^{n})$.
\end{thm}
In order to prove that, we use a very classical pseudo-differential
approach. We follow closely the ideas of L. Hörmander \cite{Hormanderboudary}
and the presentation of J. Chazarain and A. Piriou \cite{ChazarainPiriou}.
First it is necessary to study the behaviour of a subclass of $SG$-pseudo-differential
operators near the boundary.

\subsubsection{\label{sub:The behaviour of pseudo}The behaviour of SG pseudo-differential
operators near the boundary.}

In this section we will always assume that $\mu>1$ and $\nu>1$. 

We are mainly concerned with the behaviour of parametrices of SG elliptic
differential operators. Let us start studying this case in order to
clarify our assumptions.

Let $a\in S_{1\nu}^{m_{1},m_{2}}(\mathbb{R}^{n}\times\mathbb{R}^{n})$
be an SG elliptic differential symbol. Hence there are constants $C>0$,
$R>0$ and $r>0$ such that $\left|a(x,\xi)\right|\ge C\left\langle x\right\rangle ^{m_{2}}\left\langle \xi\right\rangle ^{m_{1}}$,
for all $\left|\left(x,\xi\right)\right|\ge R$, and such that all
the roots of the polynomial $z\in\mathbb{C}\mapsto a(x,\xi',z)$ lie
in some ball of radius $r\left\langle \xi'\right\rangle $.

Using the definition of $S_{1\nu}^{m_{1},m_{2}}(\mathbb{R}^{n}\times\mathbb{R}^{n})$,
it is clear that, if $\left|\left(x,\xi\right)\right|\ge R$, then
there exist constants $C>0$ and $D>0$ such that
\[
\left|\partial_{x}^{\beta}\partial_{\xi}^{\alpha}\left(\frac{1}{a(x,\xi)}\right)\right|\le CD^{\left|\alpha\right|+\left|\beta\right|}\alpha!\left(\beta!\right)^{\nu}\left\langle x\right\rangle ^{-m_{2}-\left|\beta\right|}\left\langle \xi\right\rangle ^{-m_{1}-\left|\alpha\right|}.
\]

If $\tau_{1}(x,\xi')$, ..., $\tau_{m_{1}}(x,\xi')$ are the roots
of the polynomial $z\in\mathbb{C}\mapsto a(x,\xi',z)$, then $a(x,\xi',z)=\tilde{a}(x,\xi')\Pi_{j=1}^{m_{1}}\left(z-\tau_{j}(x,\xi')\right)$,
for some function $\tilde{a}$. Using the SG ellipticity property
for $a(x,\xi',0)$ and the fact that $\left|\tau_{j}(x,\xi')\right|\le r\left\langle \xi'\right\rangle $
for all $j$, we conclude that there exists a constant $\tilde{C}>0$
such that, for all $\left|\left(x,\xi'\right)\right|\ge R$, $\left|\tilde{a}(x,\xi')\right|\ge\tilde{C}\left\langle x\right\rangle ^{m_{2}}$.

Let $\tilde{R}\ge r$. Then if $z=\tilde{R}\left\langle \xi'\right\rangle e^{i\theta}$,
then $\left|z-\tau_{j}(x,\xi')\right|\ge\left(\tilde{R}-r\right)\left\langle \xi'\right\rangle $,
for all $j$. Hence there exists a constant $C>0$ such that 
\[
\left|a\left(x,\xi',\tilde{R}\left\langle \xi'\right\rangle e^{i\theta}\right)\right|\ge C\left\langle x\right\rangle ^{m_{2}}\left\langle \xi'\right\rangle ^{m_{1}},
\]
for all $\left|\left(x,\xi'\right)\right|\ge R$. As $a\left(x,\xi',\tilde{R}\left\langle \xi'\right\rangle e^{i\theta}\right)$
is always different from zero, the above inequality holds for all
$(x,\xi')$ for some constant $C>0$. We conclude that there exist
$C>0$ and $D>0$ such that 
\[
\left|\left(\partial_{x}^{\beta}\partial_{\xi}^{\alpha}\frac{1}{a}\right)\left(x,\xi',\tilde{R}\left\langle \xi'\right\rangle e^{i\theta}\right)\right|\le CD^{\left|\alpha\right|+\left|\beta\right|}\alpha!\left(\beta!\right)^{\nu}\left\langle x\right\rangle ^{-m_{2}-\left|\beta\right|}\left\langle \xi'\right\rangle ^{-m_{1}-\left|\alpha\right|},\,\,\forall(x,\xi')\in\mathbb{R}^{n}\times\mathbb{R}^{n-1}.
\]

If $b\in S_{\mu\nu}^{-m_{1},-m_{2}}(\mathbb{R}^{n}\times\mathbb{R}^{n})$,
$\mu>1$, is a parametrix of $a$ and $\chi\in G^{\min\{\mu,\nu\}}(\mathbb{R}^{n}\times\mathbb{R}^{n})$
is a function that is zero, if $\left|(x,\xi)\right|\le R$, and equal
to $1$, if $\left|(x,\xi)\right|\ge2R$, then $b\sim\sum_{j=0}^{\infty}b_{-m_{1}-j,-m_{2}-j}$,
where
\[
b_{-m_{1},-m_{2}}=\chi\frac{1}{a},
\]
\[
b_{-m_{1}-j,-m_{2}-j}=\sum_{k+\left|\alpha\right|=j,\, k<j}\frac{1}{\alpha!}\left(D_{\xi}^{\alpha}b_{-m_{1}-k,-m_{2}-k}\right)\left(\partial_{x}^{\alpha}a\right)b_{-m_{1},-m_{2}},\,\,\, j\ge1.
\]

Using the above estimates, we conclude that $b$ must satisfy the
following conditions, which we will call assumption (A):
\begin{defn}
\label{def:Assumption A} We say that a symbol $a\in S_{\mu\nu}^{m_{1},m_{2}}(\mathbb{R}^{n}\times\mathbb{R}^{n})$,
$m_{1}\in\mathbb{Z}$ and $m_{2}\in\mathbb{R}$, satisfies the assumption
(A) if there are rational functions of $\xi$, $(x,\xi)\mapsto a_{m_{1}-j,m_{2}-j}(x,\xi)$,
and constants $C>0$, $D>0$, $B>0$ and $r>0$, with $B>r>0$, such
that

\vspace{0,2 cm}

1) For all $\left|(x,\xi)\right|\ge B$, the following holds: 
\[
\left|\left(\partial_{x}^{\beta}\partial_{\xi}^{\alpha}a_{m_{1}-j,m_{2}-j}\right)(x,\xi)\right|\le CD^{\left|\alpha\right|+\left|\beta\right|+2j}\left(j!\beta!\right)^{\nu}\alpha!\left\langle x\right\rangle ^{m_{2}-j-\left|\beta\right|}\left\langle \xi\right\rangle ^{m_{1}-j-\left|\alpha\right|}.
\]

\vspace{0,2 cm}

2) If $z_{0}$ is the pole of the function $z\in\mathbb{C}\mapsto a_{m_{1}-j,m_{2}-j}(x,\xi',z)$,
then $\left|z_{0}\right|\le r\left\langle \xi'\right\rangle $. If
this pole is real, $z_{0}\in\mathbb{R}$, then $\left|\left(x,\xi',z_{0}\right)\right|\le B$.
Moreover, for all $(x,\xi')\in\mathbb{R}^{n}\times\mathbb{R}^{n-1}$:
\[
\left|\left(\partial_{x}^{\beta}\partial_{\xi}^{\alpha}a_{m_{1}-j,m_{2}-j}\right)\left(x,\xi',B\left\langle \xi'\right\rangle e^{i\theta}\right)\right|\le CD^{\left|\alpha\right|+\left|\beta\right|+2j}\left(j!\beta!\right)^{\nu}\alpha!\left\langle x\right\rangle ^{m_{2}-j-\left|\beta\right|}\left\langle \xi'\right\rangle ^{m_{1}-j-\left|\alpha\right|}.
\]

\vspace{0,2 cm}

3) For each $M\in\mathbb{N}_{0}$ and $\left(x,\xi\right)\in\mathbb{R}^{n}\times\mathbb{R}^{n}$
such that $\max\left\{ \left\langle x\right\rangle ,\left\langle \xi\right\rangle \right\} \ge BM^{\mu+\nu-1}$,
we define the function 
\[
r_{m_{1}-M,m_{2}-M}(x,\xi):=a(x,\xi)-\sum_{j=0}^{M-1}a_{m_{1}-j,m_{2}-j}(x,\xi).
\]

These functions satisfy, for $\left(x,\xi\right)\in\mathbb{R}^{n}\times\mathbb{R}^{n}$
such that $\max\left\{ \left\langle x\right\rangle ,\left\langle \xi\right\rangle \right\} \ge BM^{\mu+\nu-1}$,
the estimate: 
\[
\left|\partial_{x}^{\beta}\partial_{\xi}^{\alpha}r_{m_{1}-M,m_{2}-M}(x,\xi)\right|\le CD^{\left|\alpha\right|+\left|\beta\right|+2M}\left(\alpha!\right)^{\mu}\left(\beta!\right)^{\nu}\left(M!\right)^{\mu+\nu-1}\left\langle x\right\rangle ^{m_{2}-M-\left|\beta\right|}\left\langle \xi\right\rangle ^{m_{1}-M-\left|\alpha\right|}.
\]

In particular, let $\chi\in G^{\min\{\mu,\nu\}}(\mathbb{R}^{n}\times\mathbb{R}^{n})$
be a function that is zero, if $\left|(x,\xi)\right|\le B$, and equal
to $1$, if $\left|(x,\xi)\right|\ge2B$, then $\sum_{j=0}^{\infty}\chi a_{m_{1}-j,m_{2}-j}\in FS_{\mu\nu}^{m_{1},m_{2}}(\mathbb{R}^{n}\times\mathbb{R}^{n})$
and $a\sim\sum_{j=0}^{\infty}\chi a_{m_{1}-j,m_{2}-j}$.\end{defn}
\begin{rem}
If $a\in S_{\mu\nu}^{m_{1},m_{2}}(\mathbb{R}^{n}\times\mathbb{R}^{n})$
and $b\in S_{\mu\nu}^{m_{1}',m_{2}'}(\mathbb{R}^{n}\times\mathbb{R}^{n})$
satisfy our assumptions, then $op(a)op(b)=op(t)+R$, where $R$ is
a $\theta-$regularizing operator for $\theta\ge\mu+\nu-1$ and $t\in S_{\mu\nu}^{m_{1}+m_{1}',m_{2}+m_{2}'}(\mathbb{R}^{n}\times\mathbb{R}^{n})$
satisfies our assumptions. This follows from the fact that if $a\sim\sum_{j=0}^{\infty}a_{m_{1}-j,m_{2}-j}$
and $b\sim\sum_{j=0}^{\infty}b_{m_{1}'-j,m_{2}'-j}$, then 
\[
t\sim\sum_{l=0}^{\infty}\left[\sum_{j+k+\left|\alpha\right|=l}\frac{1}{\alpha!}\left(\partial_{\xi}^{\alpha}a_{m_{1}-j,m_{2}-j}\right)\left(D_{x}^{\alpha}b_{m_{1}'-k,m_{2}'-k}\right)\right].
\]

\end{rem}
There are some important consequences of the assumption (A). Let us
define three paths in $\mathbb{C}$. The first one is defined as
\[
\Gamma_{\xi'}:=\left\{ B\left\langle \xi'\right\rangle e^{i\theta};\,0\le\theta\le\pi\right\} .
\]

The second one is defined as
\[
\Gamma_{\xi',M}:=\gamma_{1}\cup\Gamma_{\xi'}\cup\gamma_{2},
\]
where $\gamma_{1}$ is the real line that starts at $BM^{\mu+\nu-1}$
and ends at $B\left\langle \xi'\right\rangle $ and $\gamma_{2}$
is the real line that starts at $-B\left\langle \xi'\right\rangle $
and ends at  $-BM^{\mu+\nu-1}$.

The third path we are going to use is
\[
\Gamma_{\left(\xi'\right),M}:=\gamma_{3}\cup\Gamma_{\xi'}\cup\gamma_{4},
\]
where $\gamma_{3}$ is the real line that starts at $\sqrt{B^{2}M^{2\left(\mu+\nu-1\right)}-\left\langle \xi'\right\rangle ^{2}}$,
if it is a real number, or at $0$, if it is not real, and ends at
 $B\left\langle \xi'\right\rangle $. The curve $\gamma_{4}$ is the
real line that starts at $-B\left\langle \xi'\right\rangle $ and
ends at  $-\sqrt{B^{2}M^{2\left(\mu+\nu-1\right)}-\left\langle \xi'\right\rangle ^{2}}$,
if it is a real number, or at $0$, if it is not real.

Now let us suppose that $v\in\mathcal{S}'(\mathbb{R}^{n})$ is a distribution
such that $\left(\xi',\xi_{n}\right)\mapsto\hat{v}(\xi)$ extends
to a continuous function from $\mathbb{R}^{n-1}\times\mathbb{H}$
to $\mathbb{C}$ and $\xi_{n}\in\overset{\circ}{\mathbb{H}}\mapsto\hat{v}(\xi',\xi_{n})$
is holomorphic for each $\xi'$ fixed. Suppose that $\left|\xi'^{\alpha'}\hat{v}(\xi',\xi_{n})\right|$ is bounded,
for all $\alpha'\in\mathbb{N}_{0}^{n-1}$ and all $(\xi',\xi_{n})\in\mathbb{R}^{n-1}\times\mathbb{H}$.
This is the case of $v=e^{-}u$, where $u\in\mathcal{S}(\mathbb{R}_{-}^{n})$,
and of $v=u\otimes\delta$, where $u\in\mathcal{S}(\mathbb{R}^{n-1})$
and $\delta$ is the delta distribution in $x_{n}$. If $a$ is a
symbol that satisfies the assumption (A), we can use Cauchy Theorem
to show that, for $M\ge m_{1}+n+1$ and $x_{n}>0$: 
\[
op(a)v(x)=
\]
\[
\frac{1}{(2\pi)^n}\int_{\left\langle \xi\right\rangle \le BM^{\mu+\nu-1}}e^{ix\xi}a(x,\xi)\hat{v}(\xi)d\xi+
\]
\[
\sum_{j=0}^{M-1}\frac{1}{(2\pi)^n}\int\left(\int_{\Gamma_{\left(\xi'\right),M}}e^{ix\xi}a_{m_{1}-j,m_{2}-j}(x,\xi',\xi_{n})\hat{v}(\xi',\xi_{n})d\xi_{n}\right)d\xi'+
\]
\[
\frac{1}{(2\pi)^n}\int_{\left\langle \xi\right\rangle \ge BM^{\mu+\nu-1}}e^{ix\xi}r_{m_{1}-M,m_{2}-M}(x,\xi)\hat{v}(\xi)d\xi.
\]
and 
\[
op(a)v(x)=
\]
\[
\frac{1}{(2\pi)^n}\int_{\mathbb{R}^{n-1}}\left(\int_{\left|\xi_{n}\right|\le BM^{\mu+\nu-1}}e^{ix\xi}a(x,\xi',\xi_{n})\hat{v}(\xi',\xi_{n})d\xi_{n}\right)d\xi'+
\]
\[
\sum_{j=0}^{M-1}\frac{1}{(2\pi)^n}\int\left(\int_{\Gamma_{\xi',M}}e^{ix\xi}a_{m_{1}-j,m_{2}-j}(x,\xi',\xi_{n})\hat{v}(\xi',\xi_{n})d\xi_{n}\right)d\xi'+
\]
\[
\frac{1}{(2\pi)^n}\int_{\mathbb{R}^{n-1}}\left(\int_{\left|\xi_{n}\right|\ge BM^{\mu+\nu-1}}e^{ix\xi}r_{m_{1}-M,m_{2}-M}(x,\xi)\hat{v}(\xi',\xi_{n})d\xi_{n}\right)d\xi'.
\]

This can be used to prove the following proposition:
\begin{prop}
\label{thm:simbolo na fronteira}Let $a\in S_{\mu\nu}^{m_{1},m_{2}}(\mathbb{R}^{n}\times\mathbb{R}^{n})$
be a symbol that satisfies the assumption $(A)$. If $u\in\mathcal{S}(\mathbb{R}^{n-1})$,
then the following operator
\[
A^{kj}\left(u\right)=\lim_{x_{n}\to0^{+}}D_{x_{n}}^{k}op(a)\left(u\otimes\delta^{(j)}\right)
\]
is well defined. Moreover $A^{kj}=op(a^{kj})$, where $a^{kj}\in S_{\mu\nu}^{m_{1}+j+k+1,m_{2}}(\mathbb{R}^{n-1}\times\mathbb{R}^{n-1})$.\end{prop}
\begin{proof}
We only need to prove it for $k=j=0$, as $D_{x_{n}}^{k}op(a)D_{x_{n}}^{j}$
is also a pseudo-differential operator whose symbol satisfies assumption
(A). The proof is similar to \cite[Lemma 1]{Erkipcommunications}.
Let us choose and fix an integer $M\ge\max\left\{ 1,m_{1}+n+1\right\} $.
Using the remark about the consequences of assumption (A) and the
path $\Gamma_{\xi',M}$, we can easily take the limit in $x_{n}$
to obtain:
\[
\lim_{x_{n}\to0^+}Au(x)=Cu(x')+R_{M}u(x')+\sum_{j=0}^{M-1}C_{j}u(x'),
\]
where

\[
Cu(x')=\frac{1}{\left(2\pi\right)^{n-1}}\int e^{ix'\xi'}\left(\frac{1}{2\pi}\int_{\left|\xi_{n}\right|\le BM^{\mu+\nu-1}}a(x',0,\xi',\xi_{n})d\xi_{n}\right)\hat{u}\left(\xi'\right)d\xi',
\]
\[
R_{M}u\left(x'\right)=\frac{1}{\left(2\pi\right)^{n-1}}\int e^{ix'\xi'}\left(\frac{1}{2\pi}\int_{\left|\xi_{n}\right|\ge BM^{\mu+\nu-1}}r_{m_{1}-M,m_{2}-M}(x',0,\xi',\xi_{n})d\xi_{n}\right)\hat{u}(\xi')d\xi',
\]
\[
C_{j}u\left(x'\right)=\frac{1}{\left(2\pi\right)^{n-1}}\int e^{ix'\xi'}\left(\frac{1}{2\pi}\int_{\Gamma_{\xi',M}}a_{m_{1}-j,m_{2}-j}(x',0,\xi',\xi_{n})d\xi_{n}\right)\hat{u}(\xi')d\xi'.
\]

Now let us note that
\[
\left|\partial_{x'}^{\beta}\partial_{\xi'}^{\alpha}\left(\int_{\left|\xi_{n}\right|\le BM^{\mu+\nu-1}}a(x',0,\xi',\xi_{n})d\xi_{n}\right)\right|\le
\]
\[
CD^{\left|\alpha\right|+\left|\beta\right|}\left(\alpha!\right)^{\mu}\left(\beta!\right)^{\nu}\left\langle x'\right\rangle ^{m_{2}-\left|\beta\right|}\int_{\left|\xi_{n}\right|\le BM^{\mu+\nu-1}}\left\langle \left(\xi',\xi_{n}\right)\right\rangle ^{m_{1}-\left|\alpha\right|}d\xi_{n}\le
\]
\[
C_{1}D_{1}^{\left|\alpha\right|+\left|\beta\right|}\left(\alpha!\right)^{\mu}\left(\beta!\right)^{\nu}\left\langle x'\right\rangle ^{m_{2}-\left|\beta\right|}\left\langle \xi'\right\rangle ^{m_{1}-\left|\alpha\right|},
\]
where $C_{1}$ and $D_{1}$ are constants that depend on $C$, $D$,
$M$, $\mu$, $\nu$ and $B$, but not on $\alpha$ or $\beta$.

If $k<-n-1$, then $\int_{\mathbb{R}}\left\langle \left(\xi',\xi_{n}\right)\right\rangle ^{k}d\xi_{n}=\left\langle \xi'\right\rangle ^{k+1}\int\left\langle \xi_{n}\right\rangle ^{k}d\xi_{n}$.
Hence
\[
\left|\partial_{x'}^{\beta}\partial_{\xi'}^{\alpha}\left(\int_{\left|\xi_{n}\right|\ge BM^{\mu+\nu-1}}r_{m_{1}-M,m_{2}-M}(x',0,\xi',\xi_{n})d\xi_{n}\right)\right|\le
\]
\[
CD^{\left|\alpha\right|+\left|\beta\right|+2M}\left(M!\right)^{\mu+\nu-1}\left(\alpha!\right)^{\mu}\left(\beta!\right)^{\nu}\left\langle x'\right\rangle ^{m_{2}-M-\left|\beta\right|}\int_{\mathbb{R}}\left\langle \left(\xi',\xi_{n}\right)\right\rangle ^{m_{1}-M-\left|\alpha\right|}d\xi_{n}\le
\]
\[
C\left(\int_{\mathbb{R}}\left\langle \xi_{n}\right\rangle ^{-n-1}d\xi_{n}\right)D^{\left|\alpha\right|+\left|\beta\right|+2M}\left(M!\right)^{\mu+\nu-1}\left(\alpha!\right)^{\mu}\left(\beta!\right)^{\nu}\left\langle x'\right\rangle ^{m_{2}-M-\left|\beta\right|}\left\langle \xi'\right\rangle ^{m_{1}+1-M-\left|\alpha\right|}.
\]

In order to study the integral in $\Gamma_{\xi',M}$, we first study
in $\Gamma_{\xi'}$. Using item 2 of assumption (A) we obtain
\[
\left|\partial_{x'}^{\beta}\partial_{\xi'}^{\alpha}\left(\int_{\Gamma_{\xi'}}a_{m_{1}-j,m_{2}-j}(x',0,\xi',\xi_{n})d\xi_{n}\right)\right|\le
\]
\[
CD^{\left|\alpha\right|+\left|\beta\right|+2j}\left(j!\right)^{\nu}\alpha!\left(\beta!\right)^{\nu}\left\langle x'\right\rangle ^{m_{2}-j-\left|\beta\right|}\left|\int_{\Gamma_{\xi'}}\left\langle \xi'\right\rangle ^{m_{1}-j-\left|\alpha\right|}d\xi_{n}\right|\le
\]
\[
\pi BCD^{2M}\left(M!\right)^{\nu}D^{\left|\alpha\right|+\left|\beta\right|}\alpha!\left(\beta!\right)^{\nu}\left\langle x'\right\rangle ^{m_{2}-j-\left|\beta\right|}\left\langle \xi'\right\rangle ^{m_{1}+1-j-\left|\alpha\right|}.
\]

Finally, using item $1$ of assumption (A), we obtain
\[
\left|\partial_{x'}^{\beta}\partial_{\xi'}^{\alpha}\left(\int_{\gamma_{1}\cup\gamma_{2}}a_{m_{1}-j,m_{2}-j}(x',0,\xi',\xi_{n})d\xi_{n}\right)\right|\le
\]
\[
CD^{\left|\alpha\right|+\left|\beta\right|+2j}\left(j!\right)^{\nu}\alpha!\left(\beta!\right)^{\nu}\left\langle x'\right\rangle ^{m_{2}-j-\left|\beta\right|}\int_{\left|\xi_{n}\right|\le\mbox{max}\left\{ B\left\langle \xi'\right\rangle ,BM^{\mu+\nu-1}\right\} }\left\langle \xi\right\rangle ^{m_{1}-j-\left|\alpha\right|}d\xi_{n}\le
\]
\[
C_{1}D_{1}^{\left|\alpha\right|+\left|\beta\right|}\alpha!\left(\beta!\right)^{\nu}\left\langle x'\right\rangle ^{m_{2}-j-\left|\beta\right|}\left\langle \xi'\right\rangle ^{m_{1}+1-j-\left|\alpha\right|},
\]
where $C_{1}$ does not depend on $\alpha$ or $\beta$. It only depends
on $C$, $D$, $M$, $\mu$, $\nu$ and $B$.
\end{proof}
Now let us prove what we call the Gelfand-Shilov transmission property.
\begin{prop}
\label{pro:gelf shilov transm property}Let $a\in S_{\mu\nu}^{m_{1},m_{2}}(\mathbb{R}^{n}\times\mathbb{R}^{n})$
be a symbol that satisfies the assumption $(A)$. Then $r^{+}op(a)e^{+}$
is a map from $\mathcal{S}_{\theta}^{\theta}(\mathbb{R}_{+}^{n})$
to $\mathcal{S}_{\theta}^{\theta}(\mathbb{R}_{+}^{n})$, where $\theta\ge\mu+\nu-1$.
This property will be called Gelfand-Shilov transmission property.
\end{prop}
The proof is based on the one given by J. Chazarain and A. Piriou
\cite[Chapter 5, Section 2]{ChazarainPiriou}.
\begin{lem}
\label{lem:estimativa l de e-u}Let $u\in\mathcal{S}_{\theta}^{\theta}(\mathbb{R}_{-}^{n})$.
Then

1) The function $\xi\in\mathbb{R}_{-}^{n}\to\widehat{e_{-}u}(\xi)\in\mathbb{C}$
extends to a continuous function $\left(\xi',\xi_{n}\right)\in\mathbb{R}^{n-1}\times\mathbb{H}\mapsto\widehat{e_{-}u}(\xi)\in\mathbb{C}$
such that $\xi_{n}\in\overset{\circ}{\mathbb{H}}\mapsto\widehat{e_{-}u}(\xi',\xi_{n})\in\mathbb{C}$
is a holomorphic function for each $\xi'$ fixed.

2) There exist constants $E>0$ and $F>0$ such that, for all $l\in\mathbb{N}_{0}$,
the following estimate holds: 
\[
\left|D_{\xi}^{\sigma}\left(\widehat{e_{-}u}\right)(\xi)\right|\le EF^{l+\left|\sigma\right|}\left(l!\right)^{\theta}\left(\sigma!\right)^{\theta}\left\langle \xi'\right\rangle ^{-l},
\]
for every $\xi:=\left(\xi',\xi_{n}\right)\in\mathbb{R}^{n-1}\times\mathbb{H}$.\end{lem}
\begin{proof}
1) By definition $\widehat{e^{-}u}(\xi):=\int_{\mathbb{R}^{n-1}}\left(\int_{-\infty}^{0}e^{-ix\xi}u(x)dx_{n}\right)dx'$.
Hence if $Im\left(\xi_{n}\right)\ge0$ and $x_{n}\le0$, then $Re\left(-ix_{n}\xi_{n}\right)\le0.$
Therefore the integral is well defined and it is straightforward to
see that it defines an analytic function of $\xi_{n}$ in the upper
half-plane for each $\xi'$ fixed.

2) Using integration by parts, we obtain
\[
\left|\xi'^{\gamma}D_{\xi}^{\sigma}\int_{\mathbb{R}_{-}^{n}}e^{-ix\xi}u(x)dx\right|\le\int_{\mathbb{R}_{-}^{n}}\left(1+\left|x\right|^{2}\right)^{-n}\left|\left(1+\left|x\right|^{2}\right)^{n}D_{x'}^{\gamma}\left(x^{\sigma}u\right)(x)\right|dx\le
\]
\[
\left(\int_{\mathbb{R}_{-}^{n}}\left(1+\left|x\right|^{2}\right)^{-n}dx\right)\left(n!\right)^{\theta}CD^{\left|\gamma\right|+\left|\sigma\right|+n}\left(\sigma!\gamma!\right)^{\theta},
\]
for some contants $C>0$ and $D>0$. The result follows from the above
estimate.\end{proof}
\begin{prop}
\label{prop:r+ op(a) e-}Let $a\in S_{\mu\nu}^{m_{1},m_{2}}(\mathbb{R}^{n}\times\mathbb{R}^{n})$
be a symbol that satisfies the assumption $(A)$ and $u\in\mathcal{S}_{\theta}^{\theta}(\mathbb{R}_{-}^{n})$,
where $\theta\ge\mu+\nu-1$. Hence $r^{+}op(a)e^{-}(u)$ belongs to
$\mathcal{S}_{\theta}^{\theta}(\mathbb{R}_{+}^{n})$.\end{prop}
\begin{proof}
Along the proof $C$, $D$, $C_{1}$, $D_{1}$, ... indicate constants
that do not depend on the multi-indices $\alpha$ and $\beta$ (and
so neither on $\gamma$ nor on $\sigma$ as defined below). Sometimes
we use the same letters to indicate different constants only to avoid
a too messy notation.

We start using integration by parts to obtain
\[
x^{\alpha}\partial_{x}^{\beta}\int e^{ix\xi}a(x,\xi)\widehat{e^{-}u}(\xi)d\xi=
\]
\[
\sum_{\gamma\le\alpha}\sum_{\sigma\le\beta}\left(\begin{array}{c}
\alpha\\
\gamma
\end{array}\right)\left(\begin{array}{c}
\beta\\
\sigma
\end{array}\right)(-1)^{\left|\alpha\right|}\int e^{ix\xi}D_{\xi}^{\gamma}\left(\left(i\xi\right)^{\sigma}\left(\partial_{x}^{\beta-\sigma}a\right)(x,\xi)\right)\left(D_{\xi}^{\alpha-\gamma}\widehat{e^{-}u}\right)(\xi)d\xi.
\]

Let us study the term $\int e^{ix\xi}D_{\xi}^{\gamma}\left(\left(i\xi\right)^{\sigma}\left(\partial_{x}^{\beta-\sigma}a\right)(x,\xi)\right)\left(D_{\xi}^{\alpha-\gamma}\widehat{e^{-}u}\right)(\xi)d\xi$.
We define 
\[
M:=\max\left\{ 0,\left|\sigma\right|+m_{1}+n+1\right\} .
\]

Using the remark about the consequences of assumption (A), we know
that 
\begin{equation}
\int e^{ix\xi}D_{\xi}^{\gamma}\left(\left(i\xi\right)^{\sigma}\left(\partial_{x}^{\beta-\sigma}a\right)(x,\xi)\right)\left(D_{\xi}^{\alpha-\gamma}\widehat{e^{-}u}\right)(\xi)d\xi=\label{eq:equacao estudada}
\end{equation}
\[
\int_{\left\langle \xi\right\rangle \le BM^{\mu+\nu-1}}e^{ix\xi}D_{\xi}^{\gamma}\left(\left(i\xi\right)^{\sigma}\left(\partial_{x}^{\beta-\sigma}a\right)(x,\xi)\right)\left(D_{\xi}^{\alpha-\gamma}\widehat{e^{-}u}\right)(\xi)d\xi+
\]
\[
\sum_{j=0}^{M-1}\int e^{ix'\xi'}\left(\int_{\Gamma_{\left(\xi'\right),M}}e^{ix_{n}\xi_{n}}D_{\xi}^{\gamma}\left(\left(i\xi\right)^{\sigma}\left(\partial_{x}^{\beta-\sigma}a_{m_{1}-j,m_{2}-j}\right)(x,\xi)\right)\left(D_{\xi}^{\alpha-\gamma}\widehat{e^{-}u}\right)(\xi)d\xi_{n}\right)d\xi'+
\]
\[
\int_{\left\langle \xi\right\rangle \ge BM^{\mu+\nu-1}}e^{ix\xi}D_{\xi}^{\gamma}\left(\left(i\xi\right)^{\sigma}\left(\partial_{x}^{\beta-\sigma}r_{m_{1}-M,m_{2}-M}\right)(x,\xi)\right)\left(D_{\xi}^{\alpha-\gamma}\widehat{e^{-}u}\right)(\xi)d\xi.
\]

The second line of Equation \ref{eq:equacao estudada} is such that
\[
\left|\int_{\left\langle \xi\right\rangle \le BM^{\mu+\nu-1}}e^{ix\xi}D_{\xi}^{\gamma}\left(\left(i\xi\right)^{\sigma}\left(\partial_{x}^{\beta-\sigma}a\right)(x,\xi)\right)\left(D_{\xi}^{\alpha-\gamma}\widehat{e^{-}u}\right)(\xi)d\xi\right|\le
\]
\[
CD^{\left|\beta\right|+\left|\gamma\right|}\left(\gamma!\right)^{\mu}(\beta-\sigma)!^{\nu}\left\langle x\right\rangle ^{m_{2}-\left|\beta-\sigma\right|}EF^{\left|\alpha-\gamma\right|}\left(\alpha-\gamma\right)!^{\theta}\int_{\left\langle \xi\right\rangle \le BM^{\mu+\nu-1}}\left\langle \xi\right\rangle ^{m_{1}-\left|\gamma\right|+\left|\sigma\right|}d\xi\le
\]
\[
C_{1}D_{1}^{\left|\beta\right|+\left|\alpha\right|}\left(\gamma!\right)^{\mu}(\beta-\sigma)!^{\nu}\left\langle x\right\rangle ^{m_{2}-\left|\beta-\sigma\right|}\left(\alpha-\gamma\right)!^{\theta}\left(BM^{\mu+\nu-1}\right)^{\left|m_{1}\right|+\left|\sigma\right|+n}\le C_{2}D_{2}^{\left|\beta\right|+\left|\alpha\right|}\left(\alpha!\right)^{\theta}(\beta!)^{\theta}\left\langle x\right\rangle ^{m_{2}}.
\]

We have used that $\left(BM^{\mu+\nu-1}\right)^{\left|m_{1}\right|+\left|\sigma\right|+n}\le CD^{\left|\sigma\right|}\left(\sigma!\right)^{\mu+\nu-1}$
and that 
\[
\left|D_{\xi}^{\gamma}\left(\left(i\xi\right)^{\sigma}\left(\partial_{x}^{\beta-\sigma}a\right)(x,\xi)\right)\right|\le CD^{\left|\gamma\right|+\left|\beta\right|}\left(\gamma!\right)^{\mu}\left(\beta-\sigma\right)!^{\nu}\left\langle x\right\rangle ^{m_{2}-\left|\beta-\sigma\right|}\left\langle \xi\right\rangle ^{m_{1}+\left|\sigma-\gamma\right|}.
\]

Let us now study the term in the fourth line of Equation \ref{eq:equacao estudada}.
We note that, in the case $M=0$, this is the only term that appears in the right
hand side of Equation \ref{eq:equacao estudada} and $r_{m_{1}-M,m_{2}-M}=a$ in that situation.

For $M\ge0$, we have 
\[
\left|\int_{\left\langle \xi\right\rangle \ge BM^{\mu+\nu-1}}e^{ix\xi}D_{\xi}^{\gamma}\left(\left(i\xi\right)^{\sigma}\left(\partial_{x}^{\beta-\sigma}r_{m_{1}-M,m_{2}-M}\right)(x,\xi)\right)\left(D_{\xi}^{\alpha-\gamma}\widehat{e^{-}u}\right)(\xi)d\xi\right|\le
\]
\[
CD^{\left|\beta\right|+\left|\gamma\right|}\left(\beta-\sigma\right)!^{\nu}\left(\gamma!\right)^{\mu}\left\langle x\right\rangle ^{m_{2}-M-\left|\beta-\sigma\right|}EF^{\left|\alpha\right|-\left|\gamma\right|}\left(\alpha-\gamma\right)!^{\theta}\int_{\mathbb{R}^{n}}\left\langle \xi\right\rangle ^{m_{1}-M+\left|\sigma\right|-\left|\gamma\right|}d\xi\le
\]
\[
C_{1}D_{1}^{\left|\alpha\right|+\left|\beta\right|}\left(\alpha!\right)^{\theta}\left(\beta-\sigma\right)!^{\nu}\left\langle x\right\rangle ^{m_{2}},
\]
where we used that $m_{1}-M+\left|\sigma\right|-\left|\gamma\right|\le-n-1$.

Finally we study the term of the third line of Equation \ref{eq:equacao estudada}.
Without loss of generality, we can suppose that $M=\left|\sigma\right|+m_{1}+n+1>0$
as $0\le j<M$.

First we study the integral in $\Gamma_{\xi'}$: 
\[
\left|\int e^{ix'\xi'}\left(\int_{\Gamma_{\xi'}}e^{ix_{n}\xi_{n}}D_{\xi}^{\gamma}\left(\left(i\xi\right)^{\sigma}\left(\partial_{x}^{\beta-\sigma}a_{m_{1}-j,m_{2}-j}\right)(x,\xi)\right)\left(D_{\xi}^{\alpha-\gamma}\widehat{e^{-}u}\right)(\xi)d\xi_{n}\right)d\xi'\right|\le
\]
\[
CD^{\left|\beta\right|+\left|\gamma\right|}\left(\beta-\sigma\right)!^{\nu}\gamma!\left(j!\right)^{\nu}\left\langle x\right\rangle ^{m_{2}-j-\left|\beta\right|+\left|\sigma\right|}\int_{\mathbb{R}^{n-1}}\left(\int_{\Gamma_{\xi'}}\left\langle \xi'\right\rangle ^{m_{1}-j+\left|\sigma\right|-\left|\gamma\right|}\left(D_{\xi}^{\alpha-\gamma}\widehat{e^{-}u}\right)(\xi)d\xi_{n}\right)d\xi'.
\]

In the above expressions, we have used item 2 of assumption (A). Using the estimate of Lemma \ref{lem:estimativa l de e-u}
with $l=m_{1}+n+1+\left|\sigma\right|-j$, we conclude that the above
expression is smaller than $C_{1}D_{1}^{\left|\beta\right|+\left|\alpha\right|}\left(\alpha!\right)^{\theta}\left(\beta!\right)^{\theta}\left\langle x\right\rangle ^{m_{2}}.$

Finally the integral in $\gamma_{3}\cup\gamma_{4}$ can be evaluated
as follows
\[
\left|\int e^{ix'\xi'}\left(\int_{\gamma_{3}\cup\gamma_{4}}e^{ix_{n}\xi_{n}}D_{\xi}^{\gamma}\left(\left(i\xi\right)^{\sigma}\left(\partial_{x}^{\beta-\sigma}a_{m_{1}-j,m_{2}-j}\right)(x,\xi)\right)\left(D_{\xi}^{\alpha-\gamma}\widehat{e^{-}u}\right)(\xi)d\xi_{n}\right)d\xi'\right|\le
\]
\[
CD^{\left|\beta\right|+\left|\gamma\right|}\left(\beta-\sigma\right)!^{\nu}\gamma!\left(j!\right)^{\nu}\left\langle x\right\rangle ^{m_{2}-j-\left|\beta\right|+\left|\sigma\right|}
\]
\begin{equation}
\left|\int\left(\int_{\left|\xi_{n}\right|\le\max\left\{ B\left\langle \xi'\right\rangle ,\sqrt{B^{2}M^{2(\mu+\nu-1)}-\left\langle \xi'\right\rangle ^{2}}\right\} }\left\langle \xi\right\rangle ^{m_{1}-j+\left|\sigma\right|-\left|\gamma\right|}\left(D_{\xi}^{\alpha-\gamma}\widehat{e^{-}u}\right)(\xi)d\xi_{n}\right)d\xi'\right|.\label{eq:integral feia}
\end{equation}

If $\max\left\{ B\left\langle \xi'\right\rangle ,\sqrt{B^{2}M^{2(\mu+\nu-1)}-\left\langle \xi'\right\rangle ^{2}}\right\} =B\left\langle \xi'\right\rangle $,
then $\left\langle \xi\right\rangle ^{m_{1}-j+\left|\sigma\right|-\left|\gamma\right|}\le CD^{m_{1}-j+\left|\sigma\right|-\left|\gamma\right|}\left\langle \xi'\right\rangle ^{m_{1}-j+\left|\sigma\right|-\left|\gamma\right|}$
in the integrand, where $C$ and $D$ are contants that depend on
$B$. Using Lemma \ref{lem:estimativa l de e-u} with $l:=m_{1}+n+1+\left|\sigma\right|-j$,
we conclude that the above integral is smaller than $C_{1}D_{1}^{\left|\beta\right|+\left|\alpha\right|}\left(\alpha!\right)^{\theta}\left(\beta!\right)^{\theta}\left\langle x\right\rangle ^{m_{2}}$.

If $\max\left\{ B\left\langle \xi'\right\rangle ,\sqrt{B^{2}M^{2(\mu+\nu-1)}-\left\langle \xi'\right\rangle ^{2}}\right\} =\sqrt{B^{2}M^{2(\mu+\nu-1)}-\left\langle \xi'\right\rangle ^{2}}$,
then $\left\langle \xi\right\rangle \le BM^{\mu+\nu-1}$ in the integrand
and
\[
\left(\int_{\left\langle \xi\right\rangle \le BM^{\mu+\nu-1}}\left\langle \xi\right\rangle ^{m_{1}-j+\left|\sigma\right|-\left|\gamma\right|}\left(D_{\xi}^{\alpha-\gamma}\widehat{e^{-}u}\right)(\xi)d\xi\right)\le
\]
\[
\left\{ \begin{array}{c}
EF^{\left|\alpha-\gamma\right|}\left(\alpha-\gamma\right)!^{\theta}\left(\int_{\left|\xi\right|\le1}d\xi\right)\left(BM^{(\mu+\nu-1)}\right)^{n+m_{1}-j+\left|\sigma\right|-\left|\gamma\right|},\,\mbox{if}\,\, m_{1}-j+\left|\sigma\right|>\left|\gamma\right|\\
EF^{\left|\alpha-\gamma\right|}\left(\alpha-\gamma\right)!^{\theta}\left(\int_{\left|\xi\right|\le1}d\xi\right)B^{n}M^{n(\mu+\nu-1)},\,\mbox{if}\,\, m_{1}-j+\left|\sigma\right|\le\left|\gamma\right|
\end{array}\right..
\]

In order to conclude that the expression of Equation \ref{eq:integral feia}
is smaller than $C_{1}D_{1}^{\left|\beta\right|+\left|\alpha\right|}\left(\alpha!\right)^{\theta}\left(\beta!\right)^{\theta}\left\langle x\right\rangle ^{m_{2}}$,
we use in the first situation, when $m_{1}-j+\left|\sigma\right|>\left|\gamma\right|$, 
that $\left(j!\right)^{\nu}M^{-j(\mu+\nu-1)}\le1$ and $M^{\left(\mu+\nu-1\right)\left(n+m_{1}+\left|\sigma\right|-\left|\gamma\right|\right)}\le CD^{\left|\sigma\right|}\left(\sigma!\right)^{\mu+\nu-1}$.
In the second situation, $\left(j!\right)^{\nu}\le CD^{\left|\sigma\right|}\left(\sigma!\right)^{\nu}$
is used.
\end{proof}

The proof of Proposition \ref{pro:gelf shilov transm property} now
follows easily:
\begin{proof}
(of Proposition \ref{pro:gelf shilov transm property}) Let $f\in\mathcal{S}_{\theta}^{\theta}(\mathbb{R}_{+}^{n})$.
Let us choose $\tilde{f}\in\mathcal{S}_{\theta}^{\theta}(\mathbb{R}^{n})$
such that $r^{+}\left(\tilde{f}\right)=f$, which exists according
to Theorem \ref{thm:teo de extensao de func gelf shilov}. Let $h\in\mathcal{S}_{\theta}^{\theta}(\mathbb{R}_{-}^{n})$
be defined as $h:=r^{-}\left(\tilde{f}\right)$. Hence
\[
r^{+}op(a)e^{+}(f)=r^{+}op(a)\left(\tilde{f}-e^{-}\left(h\right)\right)=r^{+}op(a)\left(\tilde{f}\right)-r^{+}op(a)e^{-}\left(h\right).
\]
By Proposition \ref{pro:continuity in gelfandshilov}, we know that
$op(a)\left(\tilde{f}\right)\in\mathcal{S}_{\theta}^{\theta}(\mathbb{R}^{n})$.
Hence $r^{+}op(a)\left(\tilde{f}\right)\in\mathcal{S}_{\theta}^{\theta}(\mathbb{R}_{+}^{n})$.
We conclude the proof using Proposition \ref{prop:r+ op(a) e-} to
obtain $r^{+}op(a)e^{-}\left(h\right)\in\mathcal{S}_{\theta}^{\theta}(\mathbb{R}_{+}^{n})$.
\end{proof}
We conclude studying Poisson operators, similar to the ones in L.
B. de Monvel \cite{boutetMonvel}, to obtain the following result:
\begin{prop}
\label{thm:Poisson}Let $a\in S_{\mu\nu}^{m_{1},m_{2}}(\mathbb{R}^{n}\times\mathbb{R}^{n})$
be a symbol that satisfies the assumption $(A)$. If $\theta\ge\mu+\nu-1$
and $v\in\mathcal{S}_{\theta}^{\theta}(\mathbb{R}^{n-1})$, then 
\[
x'\in\mathbb{R}^{n-1}\mapsto r^{+}\left(op(a)\left(v(x')\otimes\delta(x_{n})\right)\right)\in\mathcal{S}_{\theta}^{\theta}(\mathbb{R}_{+}^{n}).
\]
\end{prop}
\begin{proof}
We just have to note that for $u\in\mathcal{S}(\mathbb{R}^{n})$ we
have 
\[
op(a)D_{n}\left(e^{+}\left(u\right)\right)=op(a)e^{+}\left(D_{n}u\right)+\frac{1}{i}op(a)\left(u(x',0)\otimes\delta(x_{n})\right).
\]
Let us choose a function $u\in\mathcal{S}_{\theta}^{\theta}(\mathbb{R}^{n})$
such that $u(x',0)=v(x')$ for $x'\in\mathbb{R}^{n-1}$. For instance,
$u(x)=v(x')e^{-x_{n}^{2}}$. Using the Gelfand-Shilov transmission
property, we conclude that $r^{+}\left(op(a)D_{n}\left(e^{+}\left(u\right)\right)\right)$
and $r^{+}\left(op(a)\left(e^{+}\left(D_{n}u\right)\right)\right)$
belong to $\mathcal{S}_{\theta}^{\theta}(\mathbb{R}_{+}^{n})$. By
the above expression, the same must hold for the expression $r^{+}\left(op(a)\left(u(x',0)\otimes\delta(x_{n})\right)\right)$.
\end{proof}
Combining all the previous results, we now prove Theorem \ref{thm:teorema principal semi plano}.
\begin{proof}
(of the Main Theorem on the half-space, Theorem \ref{thm:teorema principal semi plano}).

We first choose $\mu>1$, such that $\nu+\mu-1<\theta$ and write
$P(x,D)$ as $P(x,D)=\sum_{j=0}^{m_{1}}P_{j}(x,D_{x'})D_{x_{n}}^{j}$,
where $P_{j}(x,D_{x'})$ is a differential operator in $D_{x'}$ of
order $\le m_{1}-j$.

We then define the function $\tilde{P}:\mathcal{S}(\mathbb{R}^{n-1})^{\oplus m_{1}}\to\mathcal{S}'(\mathbb{R}^{n})$
as
\[
\tilde{P}\left(v_{0},...,v_{m_{1}-1}\right)=\frac{1}{i}\sum_{l=0}^{m_{1}-1}\sum_{j=0}^{m_{1}-l-1}P_{j+l+1}(x',0,D')v_{l}\otimes\delta^{(j)}.
\]

Hence if $\gamma:\mathcal{S}\left(\mathbb{R}_{+}^{n}\right)\to\mathcal{S}\left(\mathbb{R}^{n-1}\right)^{\oplus m_{1}}$
is the function given by $\gamma\left(u\right)=\left(\gamma_{0}\left(u\right),...,\gamma_{m_{1}-1}\left(u\right)\right)$,
we conclude that, if $u\in\mathcal{S}(\mathbb{R}_{+}^{n})$, then
\begin{equation}
P\left(e^{+}u\right)=e^{+}P\left(u\right)+\frac{1}{i}\sum_{l=0}^{m_{1}-1}\sum_{j=0}^{m_{1}-l-1}P_{j+l+1}(x',0,D')\gamma_{l}(u)\otimes\delta^{(j)}=e^{+}P\left(u\right)+\tilde{P}\gamma(u).\label{eq:P(e^+u)}
\end{equation}

Now let $u$ be the solution of Equation \ref{eq:Equacao principal no semi plano}.
We know that $u\in\mathcal{S}(\mathbb{R}_{+}^{n})=\cap_{(s,t)\in\mathbb{R}^{2}}H^{s,t}(\mathbb{R}_{+}^{n})$,
due to the SG-ellipticity of the problem and the fact that $f\in\mathcal{S}(\mathbb{R}_{+}^{n})$
and $g_{j}\in\mathcal{S}(\mathbb{R}^{n-1})$ for all $j$, see \cite{Erkiplecturenotes}.
Let $b\in S_{\mu\nu}^{-m_{1},-m_{2}}(\mathbb{R}^{n}\times\mathbb{R}^{n})$
be a parametrix of $a$. If we apply $r^{+}Q$, where $Q=op(b)$,
to both sides of Equation \ref{eq:P(e^+u)} and use that $QP=I+R$,
where $R$ is a $\theta$-regularizing operator, we obtain:
\begin{equation}
u=r^{+}Qe^{+}\left(f\right)+r^{+}Q\tilde{P}\gamma(u)-r^{+}R\left(e^{+}u\right).\label{eq:u em termos de Q P e R}
\end{equation}

Applying $\gamma$ to the above equation, we conclude that $\gamma\left(u\right)$
must satisfy
\begin{equation}
\begin{array}{c}
\left(I-\gamma Q\tilde{P}\right)\gamma(u)=\gamma\left(Qe^{+}\left(f\right)\right)-\gamma\left(R\left(e^{+}u\right)\right)\\
B\gamma\left(u\right)=g
\end{array},\label{eq:I-Q e B}
\end{equation}
where $g=\left(g_{1},...\,,g_{r}\right)$ and $\left(B\gamma\left(u\right)\right)_{j}=\sum_{k=0}^{m_{1}-1}B_{j,k}\left(x',D'\right)\gamma_{k}\left(u\right)$.

Explicitly, this means that
\[
\begin{array}{c}
\gamma(u)-\frac{1}{i}\sum_{l=0}^{m_{1}-1}\sum_{j=0}^{m_{1}-l-1}\gamma\left(Q\left(P_{j+l+1}(x',0,D')\gamma_{l}\left(u\right)\otimes\delta^{(j)}\right)\right)=\gamma\left(Qe^{+}\left(f\right)\right)-\gamma\left(R\left(e^{+}u\right)\right)\\
\sum_{k=0}^{m_{1j}}B_{j,k}(x',D')\gamma_{k}\left(u\right)=g_{j},\,\, j=1,...,r
\end{array}.
\]

According to Proposition \ref{thm:simbolo na fronteira}, the function
$Q^{kj}:\mathcal{S}(\mathbb{R}^{n-1})\to\mathcal{S}(\mathbb{R}^{n-1})$
defined as $Q^{kj}\left(v\right):=\gamma_{k}\left(Q\left(v\otimes\delta^{(j)}\right)\right)$
defines a pseudo-differential operator with symbol in $S_{\mu\nu}^{-m_{1}+j+k+1,-m_{2}}(\mathbb{R}^{n-1}\times\mathbb{R}^{n-1})$.
We now define the following functions
\[
U_{j}=\left\langle D'\right\rangle ^{-j}\gamma_{j}\left(u\right),\, F^{j}=\left\langle D'\right\rangle ^{-j}\gamma_{j}\left(Qe^{+}\left(f\right)-R\left(e^{+}u\right)\right)\,\mbox{and}\, G^{j}=\left\langle x'\right\rangle ^{-m_{2j}}\left\langle D'\right\rangle ^{-m_{1j}}g_{j},
\]
and operators
\[
\overline{Q}=\left(\frac{1}{i}\sum_{j=0}^{m_{1}-l-1}\left\langle D'\right\rangle ^{-k}Q^{kj}P_{j+l+1}(x',D')\left\langle D'\right\rangle ^{l}\right)_{k,l}\,\mbox{and}\,\overline{B}=\left(\left\langle x'\right\rangle ^{-m_{2k}}\left\langle D'\right\rangle ^{-m_{1k}}B_{k,l}(x',D)\left\langle D'\right\rangle ^{l}\right)_{k,l}.
\]

Using the Gelfand-Shilov Transmission Property, Proposition \ref{pro:gelf shilov transm property},
and the fact that $R$ is $\theta$-regularizing operator, we conclude
that $F^{j}\in\mathcal{S}_{\theta}^{\theta}(\mathbb{R}^{n-1})$. As
$g_{j}\in\mathcal{S}_{\theta}^{\theta}(\mathbb{R}^{n-1})$, using
Proposition \ref{pro:continuity in gelfandshilov}, we conclude that
$G_{j}\in\mathcal{S}_{\theta}^{\theta}(\mathbb{R}^{n-1})$, for all
$j$.

Equation \ref{eq:I-Q e B} is equivalent to
\[
\left(\begin{array}{c}
I-\overline{Q}\\
\overline{B}
\end{array}\right)\left(\begin{array}{c}
U_{0}\\
\vdots\\
U_{m_{1}-1}
\end{array}\right)=\left(\begin{array}{c}
F_{0}\\
\vdots\\
F_{m_{1}-1}\\
G_{0}\\
\vdots\\
G_{r-1}
\end{array}\right).
\]

The operator $\left(\begin{array}{c}
I-\overline{Q}\\
\overline{B}
\end{array}\right)$ is a pseudo-differential operator, whose symbol belongs to 
\[
S_{\mu\nu}^{0,0}\left(\mathbb{R}^{n-1}\times\mathbb{R}^{n-1},\mathcal{B}\left(\mathbb{C}^{m_{1}},\mathbb{C}^{\frac{3}{2}m_{1}}\right)\right).
\]

As a consequence of SG-Lopatinski-Shapiro condition \cite[Lemma 2]{Erkipcommunications},
this operator is a left elliptic pseudo-differential operator. Therefore
by Theorem \ref{thm:Regularity}, $U_{j}\in\mathcal{S}_{\theta}^{\theta}(\mathbb{R}^{n-1})$
for all $j$ and $\gamma_{j}\left(u\right)=\left\langle D'\right\rangle ^{j}U_{j}\in\mathcal{S}_{\theta}^{\theta}(\mathbb{R}^{n-1})$.
Using the Gelfand-Shilov Transmission Property, Proposition \ref{pro:gelf shilov transm property},
we conclude that $r^{+}Qe^{+}(f)\in\mathcal{S}_{\theta}^{\theta}(\mathbb{R}_{+}^{n})$.
The properties of the Poisson operator, Proposition \ref{thm:Poisson},
imply that $r^{+}Q\tilde{P}\gamma(u)\in\mathcal{S}_{\theta}^{\theta}(\mathbb{R}^{n})$.
As $R$ is a $\theta$-regularizing operator, the result follows then
from Equation \ref{eq:u em termos de Q P e R}.
\end{proof}
We note that, in the previous proof, we only need the left ellipticity
of $\left(\begin{array}{c}
I-\overline{Q}\\
\overline{B}
\end{array}\right)$. Hence the result holds for even more general boundary value problems
operators then just the ones that satisfy the SG-Lopatinski-Shapiro
condition.

{\section*{Acknowledgements}  The author would like to thank Professor Jorge G. Hounie and Professor Elmar Schrohe for fruitful discussions. We would also like to thank Professors G. Hoepfner, R. F. Barostichi, J. R. Santos Filho from UFSCar and L. Rodino from Torino for suggesting references.

\bibliographystyle{amsplain}

\end{document}